\documentclass{amsart}
 
\usepackage{amsmath}
\usepackage{amssymb}
\usepackage{amsfonts}

\usepackage{alltt}
\usepackage{amsthm}
\usepackage{bm}
\usepackage{mathrsfs}
\usepackage{graphicx}
\usepackage{color}
\usepackage{mathtools}
\usepackage[T1]{fontenc}
\usepackage[utf8]{inputenc}
\usepackage[colorlinks]{hyperref}

\usepackage{graphicx} 
\usepackage{pgfplots}

\newcommand{\C}{\mathbb C}
\newcommand{\R}{\mathbb R}
\newcommand{\N}{\mathbb N}
\newcommand{\Z}{\mathbb Z}

\newcommand{\de}{\, d}

\newcommand{\re}{\operatorname{Re}}
\newcommand{\im}{\operatorname{Im}}
\newcommand{\ve}{\varepsilon}
\newcommand{\hdot}{\bm{\ldotp}}

\newcommand{\myconj}{\ast}

\newcommand{\musw}{\mu_k^{\mathrm{sl}}}
\newcommand{\mudw}{\mu_k^{\mathrm{dl}}}
\newcommand{\sech}{\operatorname{sech}}

\numberwithin{equation}{section}

\newtheorem{thm}{Theorem}[section]
\newtheorem{cor}[thm]{Corollary}
\newtheorem{lem}[thm]{Lemma}
\newtheorem{prop}[thm]{Proposition}

\theoremstyle{definition}
\newtheorem{dfn}[thm]{Definition}

\theoremstyle{remark}
\newtheorem{rem}[thm]{Remark}
\newtheorem*{rmk}{Remark}

\makeatletter
\let\save@mathaccent\mathaccent
\newcommand*\if@single[3]{%
  \setbox0\hbox{${\mathaccent"0362{#1}}^H$}%
  \setbox2\hbox{${\mathaccent"0362{\kern0pt#1}}^H$}%
  \ifdim\ht0=\ht2 #3\else #2\fi
  }
\newcommand*\rel@kern[1]{\kern#1\dimexpr\macc@kerna}
\newcommand*\widebar[1]{\@ifnextchar^{{\wide@bar{#1}{0}}}{\wide@bar{#1}{1}}}
\newcommand*\wide@bar[2]{\if@single{#1}{\wide@bar@{#1}{#2}{1}}{\wide@bar@{#1}{#2}{2}}}
\newcommand*\wide@bar@[3]{%
  \begingroup
  \def\mathaccent##1##2{%
    \let\mathaccent\save@mathaccent
    \if#32 \let\macc@nucleus\first@char \fi
    \setbox\z@\hbox{$\macc@style{\macc@nucleus}_{}$}%
    \setbox\tw@\hbox{$\macc@style{\macc@nucleus}{}_{}$}%
    \dimen@\wd\tw@
    \advance\dimen@-\wd\z@
    \divide\dimen@ 3
    \@tempdima\wd\tw@
    \advance\@tempdima-\scriptspace
    \divide\@tempdima 10
    \advance\dimen@-\@tempdima
    \ifdim\dimen@>\z@ \dimen@0pt\fi
    \rel@kern{0.6}\kern-\dimen@
    \if#31
      \overline{\rel@kern{-0.6}\kern\dimen@\macc@nucleus\rel@kern{0.4}\kern\dimen@}%
      \advance\dimen@0.4\dimexpr\macc@kerna
      \let\final@kern#2%
      \ifdim\dimen@<\z@ \let\final@kern1\fi
      \if\final@kern1 \kern-\dimen@\fi
    \else
      \overline{\rel@kern{-0.6}\kern\dimen@#1}%
    \fi
  }%
  \macc@depth\@ne
  \let\math@bgroup\@empty \let\math@egroup\macc@set@skewchar
  \mathsurround\z@ \frozen@everymath{\mathgroup\macc@group\relax}%
  \macc@set@skewchar\relax
  \let\mathaccentV\macc@nested@a
  \if#31
    \macc@nested@a\relax111{#1}%
  \else
    \def\gobble@till@marker##1\endmarker{}%
    \futurelet\first@char\gobble@till@marker#1\endmarker
    \ifcat\noexpand\first@char A\else
      \def\first@char{}%
    \fi
    \macc@nested@a\relax111{\first@char}%
  \fi
  \endgroup
}
\makeatother

\makeatletter
\def\@setauthors{%
  \begingroup
  \def\thanks{\protect\thanks@warning}%
  \trivlist
  \centering\large \@topsep30\p@\relax
  \advance\@topsep by -\baselineskip
  \item\relax
  \author@andify\authors
  \def\\{\protect\linebreak}%
  \authors%
  \ifx\@empty\contribs
  \else
    ,\penalty-3 \space \@setcontribs
    \@closetoccontribs
  \fi
  \endtrivlist
  \endgroup
}
\def\@settitle{\begin{center}%
  \baselineskip14\p@\relax
    \normalfont\LARGE
  \@title
  \end{center}%
}
\makeatother

\pgfplotsset{compat=newest}

\begin{document}

\title[Complex eigenvalue splitting for the Dirac operator]
{Complex eigenvalue splitting for the Dirac operator}

\date{\today}

\author{Koki Hirota}
\address[Koki Hirota]{Department of Mathematical Sciences, Ritsumeikan University, Ku\-sa\-tsu, 525-8577, Japan}
\email{hirota@fc.ritsumei.ac.jp}

\author{Jens Wittsten}
\address[Jens Wittsten]{Centre for Mathematical Sciences, Lund University, Box 118, SE-221 00 Lund, Sweden, and Department of Engineering, University of Bor{\aa}s, SE-501 90 Bor{\aa}s, Sweden}
\email{jens.wittsten@math.lu.se}



\begin{abstract}
We analyze the eigenvalue problem for the semiclassical Dirac (or Zakharov-Shabat) operator on the real line with general analytic potential. We provide Bohr-Sommerfeld quantization conditions near energy levels where the potential exhibits the characteristics of a single or double bump function. From these conditions we infer that near energy levels where the potential (or rather its square) looks like a single bump function, all eigenvalues are purely imaginary. For even or odd potentials we infer that near energy levels where the square of the potential looks like a double bump function, eigenvalues split in pairs exponentially close to reference points on the imaginary axis. For even potentials this splitting is vertical and for odd potentials it is horizontal, meaning that all such eigenvalues are purely imaginary when the potential is even, and no such eigenvalue is purely imaginary when the potential is odd.
\end{abstract}

\maketitle

\section{Introduction}
Consider the eigenvalue problem
\begin{equation}
\label{evp}
P(h)u=\lambda u
\end{equation}
on the real line for the Dirac (or Zakharov-Shabat) operator given by the $2\times 2$ non-selfadjoint system 
$$
P(h)=\begin{pmatrix} -hD_x & iV(x)\\ iV(x) & hD_x\end{pmatrix},\quad D_x=-i\partial/\partial x,
$$
where $u$ is a column vector, $h$ a small positive parameter, $\lambda$ a spectral parameter, and $V$ a real-valued analytic function on $\R$. Solving \eqref{evp} constitutes an essential step in the treatment of many important nonlinear evolution equations by means of the inverse scattering transform, including the focusing nonlinear Schr{\"o}dinger (NLS) equation, the sine-Gordon equation and the modified Korteweg-de~Vries equation \cite{drazin1989solitons}. Among the numerous applications of these equations are nonlinear wave propagation in plasma physics, nonlinear fiber optics, hydrodynamics and astrophysics.

The operator $P(h)$ is the massless Dirac operator on the real line with anti-selfadjoint potential. In the selfadjoint case, resonances have been studied in various settings by many authors, see for example \cite{iantchenko2014resonances} for a historical account of the massless case, and \cite{iantchenko2014resonancesmassive} for the massive case.
Certain types of massless Dirac operators have also been shown to be effective models for twistronics, such as in Twisted Bilayer Graphene (TBG), see e.g.~\cite{tarnopolsky2019origin}. In fact, the twist-angles producing superconductive properties in TBG can be characterized in terms of the spectrum of the related Dirac operator \cite{becker2020mathematics,becker2020spectral}, which in this case is highly non-normal.

Here, we shall focus on the connection between $P(h)$ and the NLS equation, which is one of the most fundamental nonlinear evolution equations in physics.
In the focusing semiclassical case one is interested in the asymptotic behavior of $\psi=\psi(t,x;h)$ in the semiclassical limit $h\to0$, where $\psi$ is the solution to the initial value problem
\begin{equation}\label{nls}
ih\frac{\partial \psi}{\partial t}+\frac{h^2}{2}\frac{\partial^2 \psi}{\partial x^2}+\lvert\psi\rvert^2\psi=0,\quad
\psi(0,x)=V(x), 
\end{equation}
and $V$ is a real-valued function independent of $h$.
In the inverse scattering method the initial data is substituted by the {\it soliton ensembles} data, defined by replacing the scattering data for $\psi(0,x)=V(x)$ with their formal WKB approximation. The focusing NLS equation \eqref{nls} is then solved with this new set of {\it $h$-dependent} initial data, and the asymptotic behavior of the obtained approximate solution is analyzed in the limit $h\to0$. 
However, it is a priori not clear how such an $h$-dependent approximation of initial data affects the behavior of $\psi$ as $h\to0$, or if it is even justified at all.  For this a rigorous semiclassical description of the spectrum of the corresponding Dirac operator $P(h)$ is required, which has so far only been provided in a few cases such as for periodic potentials by Fujii{\'e} and Wittsten \cite{fujiie2018quantization}, and for bell-shaped, even potentials by Fujii{\'e} and Kamvissis \cite{fujiie2019semiclassical}. Both of the mentioned articles employ the {\it exact WKB method} which we describe in Section \ref{section:preliminaries} below. For an in-depth discussion on the necessity (as well as effects) of a precise description of the semiclassical spectral data of $P(h)$ in the context of inverse scattering and the focusing NLS equation we refer to the second paper mentioned above.

The interest in the spectrum of the operator $P(h)$ and its relatives dates back to Zakharov and Shabat \cite{shabat1972exact}. Since $P(h)$ is not selfadjoint the eigenvalues are not expected to be real in general. These complex eigenvalues directly determine the energy and speed of the soliton (solitary wave) solutions of \eqref{nls}; the energy, or amplitude, given by the imaginary part and the speed by the real part of the eigenvalue. Early on it was realized that there are examples of potentials $V(x)$ for which all the complex eigenvalues are in fact purely imaginary, thus giving rise to soliton pulses with zero velocity in the considered frame of reference. (In the defocusing case, obtained from \eqref{nls} by changing sign of the nonlinear term, no such solutions exist in general. In fact, the corresponding Dirac operator is then selfadjoint, and the first author has shown that  it has real spectrum even under small non-selfadjoint perturbations \cite{Hirota2017Real}.) 
In 1974, Satsuma and Yajima \cite{Satsuma1974nonlinear} studied $P(h)$ with $V(x) = V_0\sech(x)$, $V_0 > 0$, and solved \eqref{evp} by reducing it to the hypergeometric equation. They found that if $h = h_{N} = V_0/N$, there are exactly $N$ purely imaginary eigenvalues $\lambda_{k}$ given by
\begin{align*}
\lambda_{k} = ih_{N}\bigg(N-k-\frac{1}{2}\bigg), \quad k = 0,\ldots, N-1.
\end{align*}
For many years thereafter, the literature was filled with erroneous statements about eigenvalues being confined to the imaginary axis whenever the potential $V$ is real-valued and symmetric. In the nonsemiclassical regime ($h=1$) the question was given rigorous consideration in a series of papers by Klaus and Shaw \cite{klaus2001influence,klaus2002eigenvalues,klaus2003eigenvalues} who established that
\begin{itemize}
\item[(a)] if $V$ is of {\it Klaus-Shaw type}, that is, a ``single-lobe'' potential defined by a non-negative, piecewise smooth, bounded $L^{1}$ function on the real line which is nondecreasing for $x<0$ and nonincreasing for $x>0$, then all eigenvalues are purely imaginary (symmetry not being a factor);
\item[(b)] there are examples of real-valued, even, piecewise constant or piecewise quadratic potentials with two or more ``lobes'' giving rise to eigenvalues that are not purely imaginary;
\item[(c)] if $V\in L^{1}$ is an odd function, there are no purely imaginary eigenvalues at all.
\end{itemize}

We shall consider these questions in the semiclassical setting and analytic category, and show that a counterpart of (a) holds for eigenvalues near $\lambda_0=i\mu_0\in i\R$ even if one only assumes that $V$ {\it locally} has the shape of a single-lobe\footnote{Here {\it lobe} is terminology adopted from Klaus and Shaw referring to a projecting or hanging part of something, like in {\it earlobe}, or the {\it lobe of a leaf}.} potential near the ``energy level'' $\mu_0$. 
We will also derive precise conditions for eigenvalues when $V$ locally has the shape of a double-lobe potential near the energy level $\mu_0$, and show that when $V$ is symmetric, this leads to an exponentially small splitting of the eigenvalues akin to the well-known splitting phenomenon observed for eigenvalues of the selfadjoint Schr{\"o}dinger operator with a double-well potential. We prove that when $V$ is even and $h>0$ is sufficiently small, this splitting is vertical from reference points on the imaginary axis; in particular, all eigenvalues are purely imaginary then. (This is in contrast to the examples in (b) which of course do not satisfy the analyticity assumption, and we believe this might help explain the confusion witnessed in the literature prior to the mentioned papers by Klaus and Shaw.) We also show that when $V$ is odd and $h>0$ is sufficiently small, the splitting is horizontal from reference points on the imaginary axis; in particular, in accordance with (c) there can be no purely imaginary eigenvalues in this case. Here we note that for fixed $h$, \eqref{evp} can be formally interpreted as a non-semiclassical Zakharov-Shabat eigenvalue problem with potential $q(x)=h^{-1}V(x)$ and spectral parameter $\zeta=h^{-1}\lambda$, so it makes sense to compare results between the two settings. In particular, the eigenvalue formation threshold $\int_{-\infty}^\infty\lvert q(x)\rvert\,dx>\pi/2$ established by Klaus and Shaw \cite{klaus2003eigenvalues} is always reached as $h\to0$. We also wish to mention that some of the examples in (b) together with the corresponding focusing NLS equation have been further analyzed by Desaix, Andersson, Helczynski, and Lisak \cite{desaix2003eigenvalues}, and Jenkins and McLaughlin \cite{jenkins2014semiclassical}, among others.

\subsection{Statement of results}
To be more precise, we shall view $P(h)$ as a densely defined operator on $L^2$ and study the eigenvalue problem \eqref{evp} for spectral parameters $\lambda=i\mu$ close to $\lambda_0=i\mu_0\in i(0,V_0)$, where $V_{0}=\max_{x \in \mathbb{R}}\lvert V(x)\rvert$, for which the potential is either a single or double lobe in a sense to be specified below. We assume that the potential satisfies the following assumptions:
\begin{itemize}
\item[$\mathrm{(i)}$] $V(x)$ is real-valued on $\R$ and analytic in a complex domain $D\subset\C$ containing an open neighborhood of the real line, and
\item[$\mathrm{(ii)}$] $\limsup_{x\to\pm\infty}\lvert V(x)\rvert < \mu_{0}$.
\end{itemize}
Examples of $D$ are tubular neighborhoods of $\R$, or more generally, domains $\{ x \in \C : \lvert \im x\rvert < \delta(\re x) \}$ where $\delta:\R\to\R_+$ is a positive continuous function which is allowed to decay as $\lvert \re x\rvert\to\infty$.
Note that the spectrum of $P(h)$ is symmetric with respect to reflection in $\R$ (as well as with respect to reflections in the imaginary axis), so it is not necessary to treat $\lambda_0\in i(-V_0,0)$ separately.
We will also not consider spectral parameters close to the real line. In fact, if (ii) is strengthened to a decay condition of the form
\begin{itemize}
\item[$\mathrm{(ii)}'$] $\lvert V(x)\rvert\le C \lvert x\rvert^{-1-d}$ for $\lvert x\rvert\gg 1$, where $C,d>0$,
\end{itemize}
then it is known that the continuous spectrum of $P(h)$ consists of the entire real axis, and that away from the origin there are no real eigenvalues. For potentials of Klaus-Shaw type satisfying $\mathrm{(ii)}'$, a precise description of the reflection coefficients as well as the eigenvalues close to zero has recently been obtained by Fujii{\'e} and Kamvissis \cite{fujiie2019semiclassical}.

Finally, it is not necessary to consider eigenvalues away from $\R\bigcup i[-V_0,V_0]$ since the spectrum of $P(h)$ accumulates on this set in the limit as $h\to0$. In fact, if $\Omega\Subset \complement( \R\bigcup i[-V_0,V_0])$ then $P(h)$ has no spectrum in $\Omega$ as long as $h$ is sufficiently small, see Dencker \cite[Section 2]{pseudosys} or Fujii{\'e} and Wittsten \cite[Proposition 2.1]{fujiie2018quantization}.
After obtaining the necessary properties in Section \ref{section:preliminaries} of the exact WKB solutions needed for our analysis, we shall therefore in Section \ref{section:single-lobe} study the spectrum of $P(h)$ near $\lambda_0=i\mu_0$ when the potential locally, near the energy level $\mu_0$, corresponds to a single lobe in the following sense.

\begin{dfn}\label{def:single-lobe}
Let $0<\mu_0<V_0$ and assume that the equation $V(x)^2 - \mu^2_0 = 0$ has exactly two real solutions $\alpha_l(\mu_0)$ and $\alpha_r(\mu_0)$ with $\alpha_l<\alpha_r$ and $V'(\alpha_\bullet) \ne 0$, $\bullet=l,r$. We then say that $V$ is a {\it single-lobe potential near $\mu_0$}.
\end{dfn}

We may without loss of generality assume that the roots of the equation $V(x)^2-\mu^2=0$ (called {\it turning points}) are roots to $V(x)=\mu_0$ (so that $V'(\alpha_l)>0$ and $V'(\alpha_r)<0$) since the case when they are roots to $V(x)=-\mu_0$ can be studied by replacing the potential $V$ with $-V$ and reducing the resulting eigenvalue problem to the original one.\footnote{In fact, if $(u,\lambda)$ solves \eqref{evp} with $V$ replaced by $-V$, then it follows that $(v,\lambda)$ with
$$
v=\begin{pmatrix*}[r]1&0\\0&-1\end{pmatrix*}u
$$
satisfies the original eigenvalue problem \eqref{evp}, since
$$
P(h)v=P(h)\begin{pmatrix*}[r]1&0\\0&-1\end{pmatrix*}u=\begin{pmatrix*}[r]1&0\\0&-1\end{pmatrix*}\begin{pmatrix}-hD_x&-iV\\-iV&hD_x\end{pmatrix}u=\begin{pmatrix*}[r]1&0\\0&-1\end{pmatrix*}\lambda u=\lambda v.
$$
This can of course also be realized by noting that if $\psi$ solves \eqref{nls} then $\tilde\psi=-\psi$ solves the NLS equation with initial condition $\tilde\psi(0,x)=-V(x)$.}
(It is not possible that one turning point is a root to $V(x)=\mu_0$ and the other to $V(x)=-\mu_0$ since this would give four solutions to $V(x)^2-\mu^2_0=0$ when $V'(\alpha_l),V'(\alpha_r)\ne0$.)
Figure \ref{fig:single-lobe} illustrates two stereotypical examples of this situation. Of course, any potential $V$ of Klaus-Shaw type is a single-lobe potential near $\mu_0\in (0,V_0)$. Note also that the turning points depend continuously (even analytically) on $\mu$ as long as the multiplicity is constant. In particular, Definition \ref{def:single-lobe} 
cannot hold at $\mu_0=0$ or $\mu_0=V_0$ (or at any local extreme values of $V$) because the situation degenerates then, which explains why these values are excluded. 
It also explains why it makes sense to say that $V$ is a single-lobe potential {\it near} $\mu_0$, since there is then an $\ve$-neighborhood $B_\ve(\mu_0)\subset\C$ around $\mu_0$ such that if $\mu\in B_\ve(\mu_0)$ then the equation $V(x)^2 - \mu^2 = 0$ has exactly two solutions $\alpha_l(\mu)$ and $\alpha_r(\mu)$ with $\re \alpha_l<\re \alpha_r$, $\re V^{\prime}(\alpha_l) > 0$ and $\re V^{\prime}(\alpha_r) < 0$.
For such $\mu$ we define the action integral
\begin{equation}\label{actionintegralsimplelobe}
I(\mu)=\int_{\alpha_l(\mu)}^{\alpha_r(\mu)}(V(t)^2-\mu^2)^{1/2}\, dt,
\end{equation}
where the determination of the square root is chosen so that $I(\mu)$ is real and positive for real $\mu$.
In this case, we prove in Section \ref{section:single-lobe} that there are constants $\ve,h_0>0$ such that if $\mu\in B_\ve(\mu_0)$ and $0<h\le h_0$ then $\lambda=i\mu$ is an eigenvalue if and only if the Bohr-Sommerfeld quantization condition 
\begin{equation}\label{intro:qcondsw}
I(\mu) = \bigg(k+\frac{1}{2}\bigg)\pi h + h^2r(\mu,h)
\end{equation}
is satisfied for some integer $k$,
see Theorem \ref{QCL_}. Here $r(\mu,h)$ is a function defined on $B_\ve(\mu_0)\times(0,h_0]$ with $r=O(1)$ as $h\to0$. In particular, if $\musw(h)$ is the unique root of $I(\mu) = (k+\frac{1}{2})\pi h$ near $\mu_0$ (where the superscript $\mathrm{sl}$ refers to {\it single lobe}), and $\lambda_k^\mathrm{sl}=i\musw$, then there is a unique eigenvalue $\lambda_k=i\mu_k$ such that
$$
\lvert\lambda_k(h)-\lambda_k^\mathrm{sl}(h)\rvert=O(h^2),
$$
see Remark \ref{rmk:EVnbh}.
We also obtain the following refinement of \cite[Theorem 2.2]{fujiie2019semiclassical} showing that for single-lobe potentials, the semiclassical eigenvalues are confined to the imaginary axis:

\begin{figure}[!t]
\centering
\includegraphics[scale=.9]{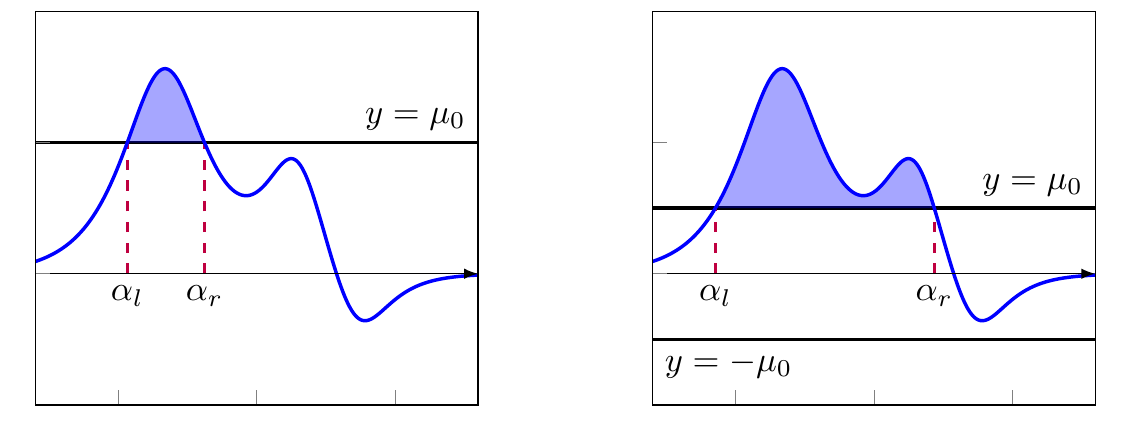}
\caption{Two examples of levels $\mu_{0}\in\R$ such that the depicted function $V$ is a single-lobe potential near $\mu_0$. The corresponding lobes are shaded blue. \label{fig:single-lobe}}
\end{figure}

\begin{thm} \label{pEV}
If $V$ is a single-lobe potential near $\mu_0=-i\lambda_0$, then there exist positive constants $h_{0}$ and $\ve$ such that the point spectrum of $P(h)$ satisfies
$$
\sigma_{p}(P(h)) \bigcap \left\{\lambda \in \mathbb{C}: \lvert\lambda-\lambda_{0}\rvert< \ve \right\} \subset i\mathbb{R}
$$
when $0 < h \leq h_{0}$.
\end{thm}

Section \ref{section:double-lobe} studies the eigenvalue problem for potentials assumed to locally have the features of a double lobe.

\begin{dfn}\label{def:double-lobe}
Let $0<\mu_0<V_0$ and assume that the equation $V(x)^2 - \mu^2_0 = 0$ has exactly four real solutions $\alpha_{l}(\mu_0)$, $\beta_{l}(\mu_0)$, $\beta_{r}(\mu_0)$ and $\alpha_r(\mu_0)$ with $\alpha_{l}< \beta_{l}< \beta_{r}<\alpha_r$ and $V^{\prime}(\alpha_\bullet),V^{\prime}(\beta_\bullet) \neq 0$, $\bullet=l,r$. We then say that $V$ is a {\it double-lobe potential near $\mu_0$.}
\end{dfn}

Figure \ref{fig:double-lobe} shows two stereotypical examples of double-lobe potentials. In the first example, $V(\beta_l)=V(\beta_r)>0$, whereas $V(\beta_l)=-V(\beta_r)>0$ in the second. 
As indicated, it suffices to consider these two situations (i.e., {\it peak-peak} and {\it peak-valley}) since the other two cases can be obtained, as for single-lobe potentials, by replacing the potential $V$ by $-V$ and reducing the corresponding eigenvalue problem to the original one. By continuity there is an $\ve>0$ such that for $\mu\in B_\ve(\mu_0)$, the equation $V(x)^2 - \mu^2 = 0$ still has exactly four solutions $\alpha_{l}(\mu)$, $\beta_{l}(\mu)$, $\beta_{r}(\mu)$ and $\alpha_r(\mu)$ with $\re\alpha_{l}< \re\beta_{l}< \re\beta_{r}<\re\alpha_r$ and the signs of $\re V'(\alpha_\bullet),\re V'(\beta_\bullet)$ unaffected.  For such $\mu$ we introduce the action integrals
\begin{equation}\label{actionintegraldoublelobe}
I_l(\mu)=\int_{\alpha_l(\mu)}^{\beta_l(\mu)}(V(t)^2-\mu^2)^{1/2}\, dt,\quad
I_r(\mu)=\int_{\beta_r(\mu)}^{\alpha_r(\mu)}(V(t)^2-\mu^2)^{1/2}\, dt,
\end{equation}
and
\begin{equation}\label{actionintegralJ}
J(\mu)=\int_{\beta_l(\mu)}^{\beta_r(\mu)}(\mu^2-V(t)^2)^{1/2}\, dt,
\end{equation}
where the determinations of the square roots are chosen in such a way that each action integral is real-valued and positive for real $\mu$. We show that there are positive constants $\ve,h_{0}$ such that if $\mu \in B_\ve(\mu_0)$ and $0<h\le h_0$ then $\lambda = i\mu$ is an eigenvalue in the case when $V(\beta_l) = \pm V(\beta_r)$ if and only if
\begin{equation}\label{eq:qc4}
\Big(e^{iI_l/h}\gamma_l+e^{-iI_l/h}\gamma_l^\myconj\Big)
\Big( e^{iI_r/h}\gamma_r+e^{-iI_r/h}\gamma_r^\myconj\Big)
\mp e^{-2J/h}\sin(I_l/h)\sin(I_r/h)=0,
\end{equation}
see Theorem \ref{qc4}. Here $\gamma_\bullet(\mu,h)$, $\bullet=l,r$, are functions defined on $B_\ve(\mu_0) \times (0,h_{0}]$ with $\gamma_\bullet=1+O(h)$ as $h\to0$, and $\ast$ denotes the operation
$$
\gamma_l^\ast(\mu)=\overline{\gamma_l(\bar \mu)},
$$
see \S\ref{ss:symmetry}.
From the quantization condition \eqref{eq:qc4} we see that modulo an exponentially small error the eigenvalues $\lambda = i\mu$ for $\mu \in B_\ve(\mu_{0})$ are given in terms of the roots to the equation
\begin{align*}
\Big(e^{iI_l/h}\gamma_l+e^{-iI_l/h}\gamma_l^\myconj\Big)
\Big( e^{iI_r/h}\gamma_r+e^{-iI_r/h}\gamma_r^\myconj\Big)
 = 0.
\end{align*}
This is equivalent to the two Bohr-Sommerfeld quantization conditions corresponding to each potential lobe, i.e.,
\begin{align*}
\frac{ \gamma_l}{ \gamma_l^\ast} e^{2iI_{l}/h} = -1, \quad \frac{ \gamma_r}{ \gamma_r^\ast} e^{2iI_{l}/h} = -1.
\end{align*}
These may be rewritten in the form
\begin{align*}
I_{l}(\mu) = \bigg(k + \frac{1}{2} \bigg)\pi h + h^{2} r_{l}(\mu,h) , \quad I_{r}(\mu) = \bigg(k + \frac{1}{2} \bigg)\pi h + h^{2} r_{r}(\mu,h),
\end{align*}
where
\begin{align*}
r_{l} = \frac{1}{2ih}\log \bigg(\frac{ \gamma_l}{ \gamma_l^\ast} \bigg), \quad r_{r} = \frac{1}{2ih}\log \bigg(\frac{ \gamma_r}{ \gamma_r^\ast}\bigg)
\end{align*}
are both bounded when $h$ tends to 0.
Thus we conclude  that the set of eigenvalues produced by a double-lobe potential is exponentially close to the union of the sets of eigenvalues produced by each potential lobe (cf.~\eqref{intro:qcondsw}).
This is a well-known fact for the Schr{\"o}dinger equation, see \cite{helffer1984multiple} and \cite{gerard1988precise}.

\begin{figure}[!t]
\centering
\includegraphics[scale=.9]{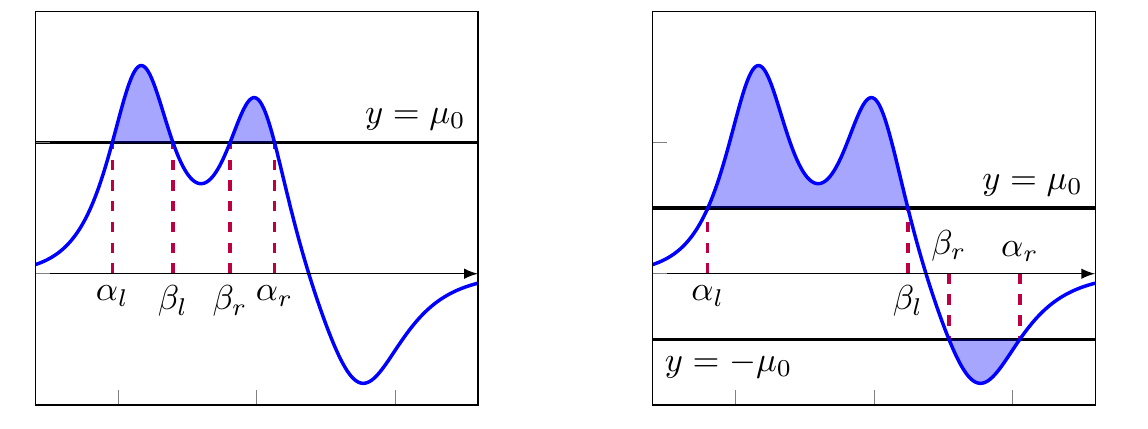}
\caption{Two examples of levels $\mu_{0}\in\R$ such that the depicted function $V$ is a double-lobe potential near $\mu_0$. The corresponding lobes are shaded blue. \label{fig:double-lobe}}
\end{figure}

\begin{rmk}
For readers familiar with the time-independent Schr{\"o}dinger equation we wish to mention that ``inside'' the lobe(s) (the projection of the shaded regions in Figures \ref{fig:single-lobe}--\ref{fig:double-lobe} onto the real axis), solutions to \eqref{evp} are oscillating, while they are exponential in character ``outside'' the lobe(s). In this sense, the lobes can thus be said to correspond to potential {\it wells} 
(rather than to {\it barriers})
in the terminology of quantum mechanics. 
\end{rmk}

Section \ref{section:symmetric} considers the special case of double-lobe potentials $V$ such that
$V(x)$ is either an even or an odd function of $x\in\R$.
If this assumption holds, the quantization condition \eqref{eq:qc4} can be rewritten in the case when $V(x) = \pm V(-x)$ as
\begin{align}
\label{experr}
 4\rho^2 \cos^2(\tilde I/h) \mp e^{-2J/h}\sin ^2( I/h)= 0,
\end{align}
see Proposition \ref{prop:qc4}.
Here, $I= I_l=I_r$, while $\tilde I= I+O(h^2)$ and $\rho =1+ O(h)$ as $h \to 0$. Thus, modulo an exponentially small error, the eigenvalues produced by each potential lobe satisfy the same 
Bohr-Sommerfeld quantization condition, namely
\begin{equation}
\label{bss}
\tilde I(\mu) = \bigg(k + \frac{1}{2} \bigg)\pi h
\end{equation}
for some integer $k=k(h)$.
If $\mudw(h)$ is the unique root of \eqref{bss} near $\mu_0$ (where the superscript $\mathrm{dl}$ stands for {\it double lobe}), it turns out that $i\mudw$ is purely imaginary.
Now, eigenvalues $\lambda=i\mu$ of the Dirac operator (where $\mu$ satisfies the quantization condition \eqref{experr}) split in pairs symmetrically about the reference points $i\mudw$.

\begin{thm}\label{splitting}
Suppose that $V$ is a double-lobe potential near $\mu_0$ such that $V(x)$ is either an even or an odd function of $x\in\R$, and let $\mudw(h)$ be the unique root of \eqref{bss} near $\mu_0$.
Then $i\mudw\in i\R$ and the two eigenvalues $i\mu_{k}^{+}(h)$, $i\mu_{k}^{-}(h)$ approximated by $i\mudw(h)$ have the following asymptotic behavior as $h \to 0$:
\begin{itemize}
\item[$(1)$]
If $V(x)$ is an even function, then
\begin{align*}
i\mu_k^\pm(h)-i\mudw(h) =  \pm i e^{-J(\mudw)/h}\bigg( \frac{h}{2I^{\prime}(\mudw)}+O(h^2)\bigg).
\end{align*}
\item[$(2)$]
If $V(x)$ is an odd function, then
\begin{align*}
i\mu_k^\pm(h)-i\mudw(h)=  \pm e^{-J(\mudw)/h}\bigg( \frac{h}{2I^{\prime}(\mudw)}+O(h^2)\bigg).
\end{align*}
\end{itemize}
Moreover, the eigenvalues split precisely vertically in the even case, 
whereas they split precisely horizontally in the odd case. Thus, for $0<h\le h_0$, all eigenvalues are purely imaginary when $V$ is even, and no eigenvalue is purely imaginary when $V$ is odd.
\end{thm}

The proof relies on the explicit exponential error term in \eqref{experr} which we obtain by using a novel method, inspired by recent work due to Mecherout, Boussekkine, Ramond and Sj{\"o}strand \cite{mecherout2016pt}, to refine the WKB analysis for the Dirac operator by introducing carefully chosen WKB solutions defined ``between'' the lobes.
As already mentioned, the results are reminiscent of the well-known splitting of eigenvalues for the linear Schr{\"o}dinger operator with a symmetric double-well potential, going back to the work of Landau and Lifshitz \cite{landau2013course} and studied mathematically by, among others, Simon \cite{simon1983instantons}, Helffer and Sj{\"o}strand \cite{helffer1984multiple} and G{\'e}rard and Grigis \cite{gerard1988precise}. This type of tunneling effect has recently also been observed for a system of semiclassical Schr{\"o}dinger operators by Assal and Fujii{\'e} \cite{assal2019eigenvalues}. For more on this topic we refer to the mentioned works and the references therein.

In the literature a common focus of study is the appearance and location of purely imaginary eigenvalues as the $L^1$ norm of the potential increases, for example by taking $q(x)=h^{-1}V(x)$ and letting $h$ decrease. Potentials of the form
$$
q(x)=h^{-1}(\sech(x-x_0)+\sech(x+x_0))
$$
consisting of two separated sech-shaped pulses have been numerically investigated by Desaix, Anderson and Lisak \cite{desaix2008eigenvalues} for different separations $x_0$. They found that at the first critical amplitude $h^{-1}=1/4$, a purely imaginary eigenvalue $\zeta_1$ appears, and for $h^{-1}<1/4$ there are no eigenvalues (consistent with the threshold of Klaus and Shaw \cite{klaus2003eigenvalues}). For small separations, $q$ behaves almost like a single-lobe potential, and the second critical amplitude $h^{-1}=3/4$ also gives rise to a purely imaginary eigenvalue. However, for larger separations such as $x_0=5$, two complex eigenvalues $\zeta_{2,3}=\pm\xi+i\eta$ with nonzero real parts are created already in the vicinity of $h^{-1}=4/10$. As the amplitude $h^{-1}$ increases, the real parts decrease while $\eta$ increases until the two eigenvalues meet and then separate along the imaginary axis (both now purely imaginary, $\zeta_2$ with increasing and $\zeta_3$ with decreasing imaginary part). As $h^{-1}$ reaches the second critical amplitude $3/4$, $\zeta_3$ is destroyed and only $\zeta_1$ and $\zeta_2$ remain. Since $\zeta=h^{-1}\lambda$, we should be able to see a similar type of behavior for semiclassical eigenvalues of $P(h)$ as $h$ decreases, which is something we hope to investigate in a future paper. Of course, as $h$ becomes sufficiently small, our results show that for a potential consisting of two separated sech-pulses, all eigenvalues $\lambda=i\mu$ are purely imaginary as long as $\mu\ne0$ is not close to a local extreme value of $V$, see Figure \ref{fig:single-double}. It would also be interesting to see if this property persists as $\mu$ tends to local extreme values of $V$; the exact WKB method does not work then so other methods are needed.

\begin{figure}[!t]
\centering
\includegraphics[scale=.9]{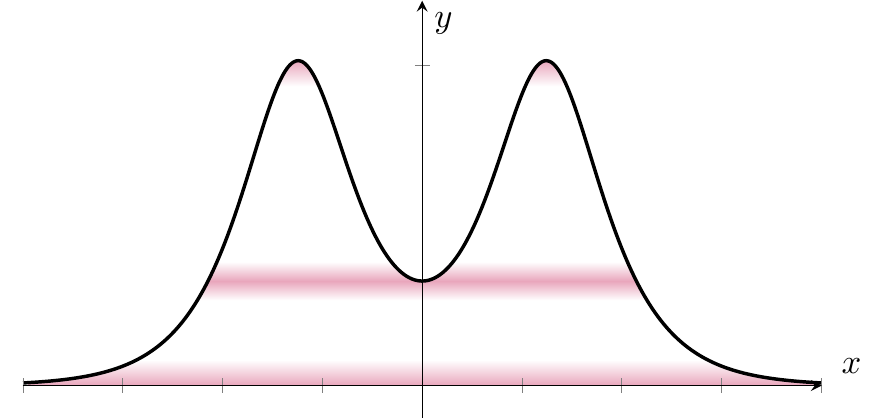}
\caption{An even potential $V(x)$ with a local minimum at $x=0$. Away from the shaded region $V$ is either a single-lobe or a double-lobe potential. For sufficiently small $h$, any eigenvalue $\lambda$ of $P(h)$ with imaginary part away from the shaded region must therefore be purely imaginary by Theorems \ref{pEV} and \ref{splitting}. \label{fig:single-double}}
\end{figure}

\section{Exact WKB analysis}\label{section:preliminaries}
\subsection{Exact WKB solutions}

Here we recall the construction of a solution of the Dirac system in a complex domain as a convergent series, known as an exact WKB solution. Such solutions were first introduced by Ecalle \cite{ecalle1984cinq} and later used by G\'erard and Grigis \cite{gerard1988precise} to study the Schr\"odinger operator. We shall follow the construction for systems due to Fujii{\'e}, Lasser and N{\'e}d{\'e}lec \cite{fujiie2009semiclassical}.

The system \eqref{evp} can be written in the form
\begin{equation}\label{system}
\frac{h}{i}\frac{du}{dx}=M(x,\lambda)u,\quad M(x,\lambda)=\begin{pmatrix} -\lambda & iV(x)\\ -iV(x) & \lambda\end{pmatrix}.
\end{equation}
Recall (see \cite{fujiie2009semiclassical}) that the exact WKB solutions of systems of type \eqref{system} are of the form
\begin{equation}\label{exactWKB}
u^\pm(x,h)=e^{\pm z(x)/h}\begin{pmatrix*}[r] 1 & 1\\ -1 & 1\end{pmatrix*}Q(z(x))\begin{pmatrix} 0 & 1\\ 1 & 0\end{pmatrix}^{(1\pm1)/2}w^\pm(x,h),
\end{equation}
where the function $z(x)$ is the complex change of coordinates
\begin{equation*}
z(x)=z(x;x_0)=i\int_{x_0}^x\sqrt{V(t)^2+\lambda^2}\de t
\end{equation*}
for some choice of phase base point $x_0$ in the strip $D$ where $V$ is assumed to be analytic, while $Q$ is the matrix valued function
\begin{equation}\label{defQ}
Q(z)=\begin{pmatrix} H(z)^{-1} & H(z)^{-1}\\ iH(z) & -iH(z)\end{pmatrix}
\quad \text{with }
H(z(x))=\bigg(\frac{iV(x)+\lambda}{iV(x)-\lambda}\bigg)^{1/4}.
\end{equation}
Here $z(x)$ and $Q(z)$ are defined on the Riemann surfaces of $(V^2+\lambda^2)^{1/2}$ and $H(z({\hdot}))$ over $D$, respectively.
These Riemann surfaces are defined by introducing branch cuts emanating from the zeros of $x\mapsto\det(M(x,\lambda))$, i.e., of $iV\pm\lambda$ (the turning points of the system \eqref{system}), see \S\ref{ss:riemann}.

The amplitude vectors $w^\pm$ in \eqref{exactWKB} are defined as the (formal) series
\begin{equation}\label{amplitude}
w^\pm(x,h)= \begin{pmatrix} w^\pm_{\mathrm{even}}(x,h)\\ w^\pm_{\mathrm{odd}}(x,h)\end{pmatrix}
=\sum_{n=0}^\infty 
\begin{pmatrix} w^\pm_{2n}(z(x))\\ w^\pm_{2n+1}(z(x))\end{pmatrix},
\end{equation}
where $w^\pm_0(z)\equiv 1$, while $w^\pm_j(z)$ for $j\ge1$ are the unique solutions to the scalar transport equations
\begin{equation}\label{transportodd}
\bigg(\frac{d}{dz}\pm\frac{2}{h}\bigg)w^\pm_{2n+1}(z)=\frac{dH(z)/dz}{H(z)}w^\pm_{2n}(z),
\end{equation}
\begin{equation}\label{transporteven}
\frac{d}{dz}w^\pm_{2n+2}(z)=\frac{dH(z)/dz}{H(z)}w^\pm_{2n+1}(z)
\end{equation}
with prescribed initial conditions $w_n^\pm(\tilde z)=0$ for some choice of amplitude base point $\tilde z=z(\tilde x)$ where $\tilde x$ is not a turning point. 
When we want to signify the dependence on the base point $\tilde z=z(\tilde x)$ we write $$
w^\pm(x,h;\tilde x)
=\begin{pmatrix} w^\pm_{\mathrm{even}}(x,h;\tilde x)\\ w^\pm_{\mathrm{odd}}(x,h;\tilde x)\end{pmatrix}
$$
for the amplitude vectors.

Recall that if $\Omega$ is a simply connected open subset of $D$ which is free from turning points then $z=z(x)$ is conformal from $\Omega$ onto $z(\Omega)$. For fixed $h>0$, the formal series \eqref{amplitude} converges uniformly in a neighborhood of the amplitude base point $\tilde x$, and $w^\pm_{\mathrm{even}}(x,h)$ and $w^\pm_{\mathrm{odd}}(x,h)$ are analytic functions in $\Omega$, see \cite[Lemma 3.2]{fujiie2009semiclassical}. As a consequence, the functions $u^\pm$ given by \eqref{exactWKB} are exact solutions of \eqref{system} and when we wish to indicate the particular choice of amplitude base point $\tilde x\in \Omega$ and phase base point $x_0\in D$ we will write $u^\pm(x;x_0,\tilde x)$.
We remark that these solutions are defined for example everywhere on $\R$, although some of the expressions involved are only defined on Riemann surfaces of $(V^2+\lambda^2)^{1/2}$ or $H(z({\hdot}))$.

For fixed $\tilde x\in\Omega$, let $\Omega_\pm$ be the set of points $x$ for which there is a path from $\tilde x$ to $x$ along which $t\mapsto \pm\re z(t)$ is strictly increasing. In other words, $x\in\Omega_\pm$ if there is a path which intersects the the level curves of $t\mapsto \re z(t)$ transversally in the appropriate direction. The level curves of $t\mapsto \re z(t)$ are called {\it Stokes lines}.

\begin{rem}\label{h-asymptotics}
For any integers $k,N\in\N$
$$
\partial^k\big(w^\pm_\mathrm{even}(x,h)-\sum_0^Nw_{2n}^\pm(z(x))\big)=O(h^{N+1}),
$$
$$
\partial^k\big(w^\pm_\mathrm{odd}(x,h)-\sum_0^Nw_{2n+1}^\pm(z(x))\big)=O(h^{N+2}),
$$
uniformly on compact subsets of $\Omega_\pm$ as $h\to0$, see \cite[Proposition 3.3]{fujiie2009semiclassical}.
In particular,
$$
w^\pm_\mathrm{even}(x,h)=1+O(h),\quad w^\pm_\mathrm{odd}(x,h)=O(h),
$$
as $h\to0$.
\end{rem}

\subsection{The Wronskian formula}

For vector-valued solutions $u$ and $v$ of \eqref{system}, let $\mathcal W(u,v)$ be the {\it Wronskian} defined by
$$
\mathcal W(u,v)(x)=\det(u(x)\ v(x)).
$$
Since the trace of the matrix $M(x,\lambda)$ is zero, it follows that $\mathcal W(u,v)$ is in fact independent of $x$. 
If $x_0$ is a phase base point in $D$ and $\tilde x,\tilde y$ are different amplitude base points in $\Omega$, a straightforward calculation shows that
\begin{multline*}
\mathcal W(u^+(x;x_0,\tilde x),u^-(x;x_0,\tilde y))
\\=-4i\big(
w_\mathrm{odd}^+(x,h;\tilde x)w_\mathrm{odd}^-(x,h;\tilde y)-
w_\mathrm{even}^+(x,h;\tilde x)w_\mathrm{even}^-(x,h;\tilde y)\big),
\end{multline*}
where the solutions $u^\pm$ are given by \eqref{exactWKB}. Recalling the initial conditions of the transport equations \eqref{transportodd}--\eqref{transporteven} and evaluating at $x=\tilde y$ we get
\begin{equation}\label{Wronskian1+}
\mathcal W(u^+(x;x_0,\tilde x),u^-(x;x_0,\tilde y))
=4iw_\mathrm{even}^+(\tilde y,h;\tilde x).
\end{equation}
We may of course also choose $x=\tilde x$, which gives 
\begin{equation}\label{Wronskian1-}
\mathcal W(u^+(x;x_0,\tilde x),u^-(x;x_0,\tilde y))
=4iw_\mathrm{even}^-(\tilde x,h;\tilde y).
\end{equation}
In particular, we see that if there is a path from $\tilde x$ to $\tilde y$ along which the function $t\mapsto \re z(t)$ is strictly increasing, then $\mathcal W(u^+(x;x_0,\tilde x),u^-(x;x_0,\tilde y))
=4i+O(h)$ as $h\to0$ by Remark \ref{h-asymptotics}, showing that such a pair of solutions is linearly independent if $h$ is sufficiently small. We also recall the Wronskian formula for pairs of solutions of the same type: 
\begin{multline*}
\mathcal W(u^\pm(x;x_0,\tilde x),u^\pm(x;x_0,\tilde y))
=-4i(-1)^{(1\pm 1)/2}e^{\pm2 z(x;x_0)/h}\\ \times\big(
w_\mathrm{even}^\pm(x,h;\tilde x) w_\mathrm{odd}^\pm(x,h;\tilde y)-
w_\mathrm{even}^\pm(x,h;\tilde y)w_\mathrm{odd}^\pm(x,h;\tilde x)\big).
\end{multline*}
Evaluating at $x=\tilde x$ gives
\begin{equation}\label{Wronskian2}
\mathcal W(u^\pm(x;x_0,\tilde x),u^\pm(x;x_0,\tilde y))
=-4i(-1)^{(1\pm 1)/2}e^{\pm2 z(\tilde x;x_0)/h}
w_\mathrm{odd}^\pm(\tilde x,h;\tilde y).
\end{equation}

\subsection{Stokes geometry}\label{ss:stokesgeometry}
We now describe the configuration of Stokes lines for single-lobe and double-lobe potentials.
\subsubsection{Single-lobe potentials} Suppose that $V$ is a single-lobe potential near $\mu_0$ and let $\mu\in B_\ve(\mu_0)$. Fix determinations of $H(z(x))$ given by \eqref{defQ} and of
\begin{equation}\label{eq:zintopbottom}
z(x;\alpha_\bullet)=z(x;\alpha_\bullet(\mu),\mu)=i\int_{\alpha_\bullet}^x(V(t)^2-\mu^2)^{1/2}\,dt,\quad \bullet=l,r,
\end{equation}
by picking branches so that $H(z(x))>0$ and $(V(x)^2-\mu^2)^{1/2}>0$ when $\alpha_l<x<\alpha_r$ and $\mu\in\R$. Note that this is in accordance with \eqref{actionintegralsimplelobe}.
The Stokes lines (level curves of $t\mapsto \re z(t;\alpha_\bullet)$) are then found by taking the union of
\begin{equation}\label{eq:defstokeslines}
	\Big\{ x \in D :\im \int_{x_0}^{x}\sqrt{V(t)^{2}-\mu^{2}}\,dt = {\rm const.} \Big\}
\end{equation}
for $x_0=\alpha_l,\alpha_r$. When $\mu$ is real it is known that there are three Stokes lines emanating from $\alpha_l\in\R$ having arguments $0, 2\pi/3,4\pi/3$, while the Stokes lines emanating from $\alpha_r\in\R$ have arguments $\pi/3, \pi,5\pi/3$, see G{\'e}rard and Grigis \cite{gerard1988precise}. We define the Riemann surfaces of $z(x)$ and $H(z(x))$ by introducing branch cuts along the Stokes line with argument $2\pi/3$ at $\alpha_l$ and the Stokes line with argument $5\pi/3$ at $\alpha_r$. For real $\mu\in B_\ve(\mu_0)$ there is a bounded Stokes line lying on $\R$ starting at $\alpha_l$ and ending at $\alpha_r$. Hence, the Stokes lines separate the complex domain $D$ into four sectors (called {\em Stokes regions}). In the top and bottom sectors the function $z(x)$ takes the form \eqref{eq:zintopbottom}. By continuing the chosen determination of $z(x)$ through rotation clockwise around the turning points (thus avoiding the branch cuts) it is easy to see that 
\begin{equation}\label{eq:zinleftright}
z(x;\alpha_\bullet)=\int_{\alpha_\bullet}^{x}(\mu^2-V(t)^2)^{1/2}\,dt
\end{equation}
for $x$ belonging to the left and right sector when $\bullet=l$ and $\bullet=r$, respectively. 
For general $\mu\in B_\ve(\mu_0)$ the picture is slightly perturbed; as $i\mu$ is rotated off the imaginary axis $\alpha_l$ and $\alpha_r$ start migrating in opposite directions along paths in the upper and lower half plane, and the bounded Stokes line connecting $\alpha_l$ and $\alpha_r$ is broken into two unbounded curves, see Figure \ref{fig:stokes_config_single}. (We refer to \cite{MR1635811} for a detailed explanation of this phenomenon.) However, for small $\ve$ the arguments of the Stokes lines at the turning points are almost unchanged so for $\mu\in B_\ve(\mu_0)$ we may still place branch cuts as described above. Note that there are now three Stokes regions around the left turning point and three around the right, and \eqref{eq:zinleftright} is still valid if interpreted in this sense. However, we will avoid introducing notation for the different Stokes regions, and simply say (informally) that $x$ is {\it near the lobe} if $x$ is not in the Stokes region to the left of $\alpha_l$ or to the right of $\alpha_r$. We also remark that if $x_0(\mu)$ is a turning point satisfying $V(x_0(\mu))=\mu$, then $x_0(-\mu)$ is also a solution to $V(x)^2-\mu^2=0$; hence the original Stokes configuration is reached again already when $i\mu$ has traversed half a circuit around the origin, see the left panel of Figure \ref{fig:migration} below.

\begin{figure}[!t]
\centering
\includegraphics[scale=.9]{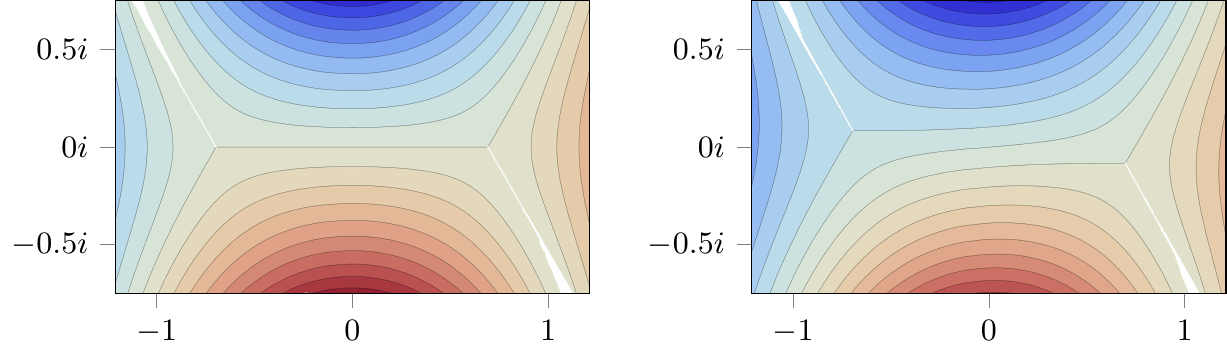}
\caption{\label{fig:stokes_config_single} 
Configuration of Stokes lines in the complex plane for the single-lobe potential $V(x)=\frac{1}{4}\sech(x)$ near $\mu_0=0.2$. The left panel describes the situation for $\mu=\mu_0$ and the right panel when $\mu=0.2+0.01i$ has been perturbed to have small positive imaginary part. The panels show the increasing value of $\re z(x;\alpha_r(\mu))$ as one travels from blue toward red regions. Here $\alpha_r(\mu)$ is the turning point on the right, so as indicated $\re z(x)$ is zero along the Stokes lines emanating from $\alpha_r$. Branch cuts are located along (the curved edges of) the white regions.}
\end{figure}

In Figure \ref{fig:stokes_config_single} we have also indicated that $\re z(x)$ increases as one travels from top to bottom and left to right, while not passing through a branch cut. This is realized in the following way: For $x$ in the regions between turning points we have by \eqref{eq:zintopbottom} and Taylor's formula that
\begin{equation}\label{eq:Taylor1}
z(x;\alpha_\bullet)-z(x_0;\alpha_\bullet)=i(x-x_0)(V(x_0)^2-\mu^2)^{1/2}(1+g_1(x)),
\end{equation}
where $g_1$ is analytic and $g_1(x_0)=0$. Since $V(x_0)>\mu_0$ if $\alpha_l(\mu_0)<x_0<\alpha_r(\mu_0)$ we see by picking $x_0$ real that the square root is approximately real when $\mu\in B_\ve(\mu_0)$, so $\re z(x)$ increases as $\im x$ decreases.
On the other hand, for $x$ in the Stokes region left of $\alpha_l$ or right of $\alpha_r$ we have by \eqref{eq:zinleftright} and Taylor's formula that
\begin{equation}\label{eq:Taylor2}
z(x;\alpha_\bullet)-z(x_0;\alpha_\bullet)=(x-x_0)(\mu^2-V(x_0)^2)^{1/2}(1+g_2(x)),
\end{equation}
where $g_2$ is analytic and $g_2(x_0)=0$. By picking $x_0\in D\cap\R$ with $\lvert\re x_0\rvert\gg1$ we have $V(x_0)\approx0$ showing that $\re z(x)$ increases as $\re x$ increases. This also shows that $\re z(x)$ is constant along lines which are essentially vertical near $\R$ when $\lvert\re x\rvert$ is large.
\subsubsection{Double-lobe potentials}
Suppose now that $V$ is a double-lobe potential near $\mu_0$ and let $\mu\in B_\ve(\mu_0)$.
Again, fix determinations of $H(z(x))$ and $z(x)$ in accordance with \eqref{actionintegraldoublelobe}--\eqref{actionintegralJ}; the obtained configuration of Stokes lines will essentially be two side-by-side copies of the configuration for single-lobe potentials with an appropriate gluing in the region between the two middle turning points $\beta_l$ and $\beta_r$.

\begin{figure}[!t]
\includegraphics[scale=.9]{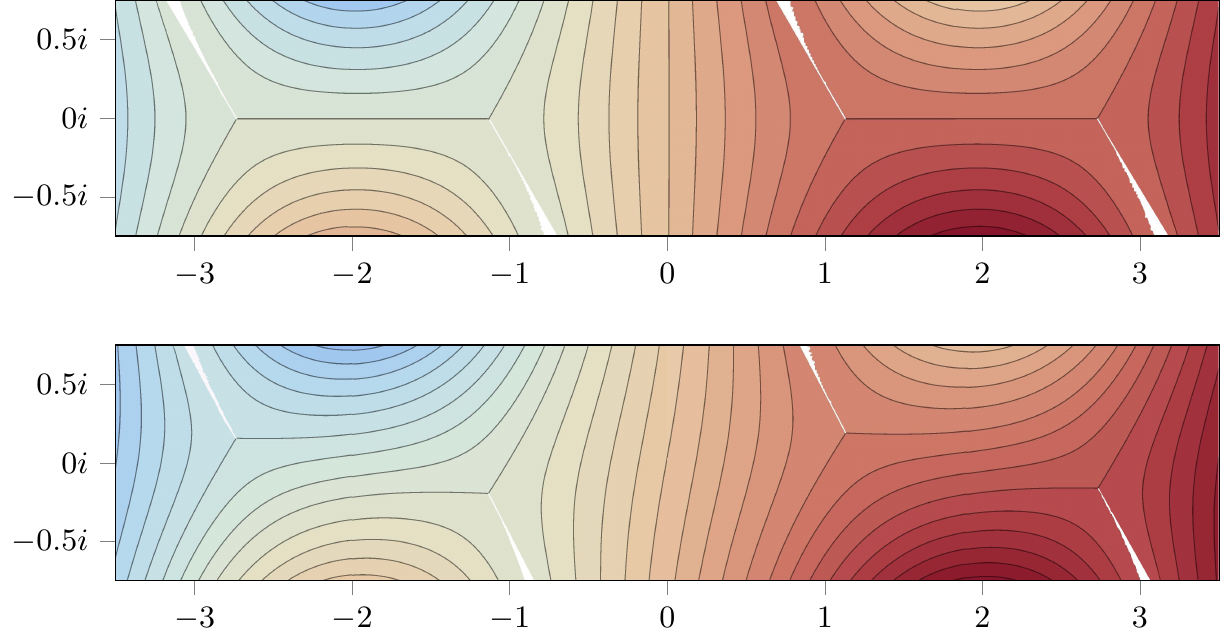}
\caption{\label{fig:stokes_config_double} 
Configuration of Stokes lines in the complex plane for the double-lobe potential $V(x)=\frac{1}{4}(\sech(x-2)+\sech(x+2))$ near $\mu_0=0.2$. The top panel describes the situation for $\mu=\mu_0$ and the bottom panel when $\mu=0.2+0.02i$ has been perturbed to have small positive imaginary part. The panels show the increasing value of $\re z(x;\beta_l(\mu))$ as one travels from blue toward red regions. Here $\beta_l(\mu)$ is the second turning point from the left, so as indicated $\re z(x)$ is zero along the Stokes lines emanating from $\beta_l$. Branch cuts are located along (the curved edges of) the white regions.}
\end{figure}

Indeed, the Stokes lines are given by the union of \eqref{eq:defstokeslines} for $x_0=\alpha_l,\beta_l,\beta_r,\alpha_r$. When $\mu$ is real there are three Stokes lines emanating from $\alpha_l$ and three from $\beta_r$ having arguments $0, 2\pi/3,4\pi/3$, while the Stokes lines emanating from $\beta_l,\alpha_r\in\R$ have arguments $\pi/3, \pi,5\pi/3$, see G{\'e}rard and Grigis \cite{gerard1988precise}. As $i\mu$ is rotated off the imaginary axis the turning points start migrating in alternating, opposing directions along paths in the upper and lower half plane, so that $\alpha_l$ moves in the direction opposite from $\beta_l$ but similar to $\beta_r$. We place branch cuts along the Stokes lines which for real $\mu$ have arguments $2\pi/3$ at $\alpha_l,\beta_r$ and the Stokes lines with arguments $5\pi/3$ at $\beta_l,\alpha_r$. Performing the same analysis as above shows that in the sectors to the left of $\alpha_l$ and to the right of $\alpha_r$, and in the intersection of the sectors to the right of $\beta_l$ and to the left of $\beta_r$ (i.e., between $\beta_l$ and $\beta_r$), $z(x;\alpha_\bullet)$ takes the form \eqref{eq:zinleftright}.  When $x$ is in the other sectors (between $\alpha_l$ and $\beta_l$ or between $\beta_r$ and $\alpha_r$), $z(x;\alpha_\bullet)$ is given by \eqref{eq:zintopbottom}, and as for single-lobe potentials we shall informally say that $x$ is {\it near the lobes} in this case. Using Taylor's formula as in \eqref{eq:Taylor1}--\eqref{eq:Taylor2} then shows that $\re z(x)$ increases as one travels from top to bottom and left to right, while not passing through a branch cut, see Figure \ref{fig:stokes_config_double}. The right panel of Figure \ref{fig:migration} shows an example of how the turning points of a double-lobe potential migrate as $i\mu$ is rotated off the imaginary axis.

\begin{figure}[!t]
\centering
\includegraphics[scale=.9]{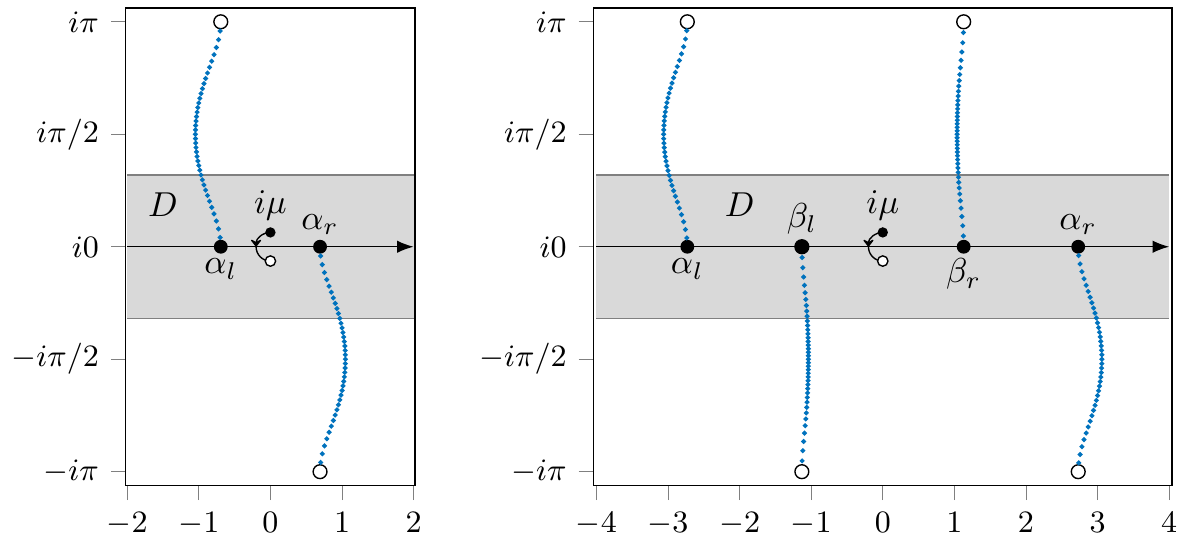}
\caption{The migration paths of turning points (solutions to $V(x)^2-\mu^2=0$) of the potentials in Figure \ref{fig:stokes_config_single} (left) and Figure \ref{fig:stokes_config_double} (right) as $i\mu$ is rotated $\pi$ radians in the positive direction from the starting value $\mu=0.2$ until $\mu=-0.2$ when the original Stokes geometry is recovered. Black dots and circles mark starting and finishing locations, respectively. Rotation in the opposite direction reverses the direction of migration. Note that since $\sech(x+i\pi)=-\sech(x)$ the turning points appear periodically in $\C$ with complex period $i\pi$ for both potentials. Examples of the domain $D$ (gray) are shown to indicate that only small rotations of $i\mu$ are of interest for the problem under consideration here.\label{fig:migration}}
\end{figure}

\subsection{The Riemann surface}\label{ss:riemann}
Let $\mathcal R(x_0,\theta)$ denote the operator acting through rotation around $x_0$ by $\theta$ radians, so that, e.g., $\mathcal R(0,\theta)x=e^{i\theta}x$. Since $V-\mu$ is analytic and $V(\alpha_l)-\mu=0$ it follows that
\begin{equation*}
V(\mathcal R(\alpha_l,2\pi k)t)-\mu=e^{2i\pi k}(V(t)-\mu),\quad k\in\Z,
\end{equation*}
i.e., when $t$ is rotated $2\pi k$ radians anticlockwise around $\alpha_l$ then $V(t)-\mu$ is rotated $2\pi k$ radians anticlockwise around the origin. (Negative $k$ results in clockwise rotation by $2\pi|k|$ radians.)
We of course have similar behavior near the other turning points of the same type, as well as for $V+\mu$ in the case when e.g.~$V(\beta_r)+\mu=0$.

\begin{dfn}\label{def:sheets}
Suppose that $V$ is a single-lobe (double-lobe) potential near $\mu_0$ and
let $y$ be a point in the upper half plane with $\re \alpha_l<\re y<\re\alpha_r$ ($\re \alpha_l<\re y<\re\beta_l$).
The point over $y$ that is obtained when rotating $y$ anticlockwise once around $\alpha_l$ will be denoted by $\hat y$, i.e.,
$$
\hat y=\mathcal R({\alpha_l},2\pi)y.
$$
More generally, the sheet reached (from the usual sheet) by entering the cut starting at $\alpha_l$ {\it from the right} will be referred to as the $\hat x$-sheet. 
The point over $y$ that is obtained when rotating $y$ clockwise once around $\alpha_l$ will be denoted by $\check y$, i.e.,
$$
\check y=\mathcal R(\alpha_l,-2\pi)y.
$$
The sheet reached (from the usual sheet) by entering the cut starting at $\alpha_l$ {\it from the left} will be referred to as the $\check x$-sheet. 
\end{dfn}

Note that this definition is in accordance with \cite[Definition 5.2]{fujiie2018quantization}. When winding this way around a turning point we always assume that the path is appropriately deformed so that it is not obstructed by other branch cuts.
Informally, we think of $\hat x$ as lying in the sheet ``above'' the usual sheet, 
and $\check x$ as lying in the sheet ``below'' the usual sheet. 
It is straightforward to check that the $\hat x$-sheet is also reached (from the usual sheet) whenever we rotate anticlockwise once around the other zeros of $V-\mu$ (i.e., around $\beta_l$, $\beta_r$ and $\alpha_r$ if $V(\beta_l)=V(\beta_r)$, 
and around $\beta_l$ if $V(\beta_l)=-V(\beta_r)$). 
Similarly, the $\check x$-sheet is reached (from the usual sheet) by rotating clockwise once around zeros of $V-\mu$. The directions are reversed when rotating around zeros of $V+\mu$, i.e., when rotating around $\beta_r$ and $\alpha_r$ if $V(\beta_l)=-V(\beta_r)$. For a proof of these facts we refer to \cite[Lemma 5.3]{fujiie2018quantization}.
We also record the following identities describing how WKB solutions are transformed when switching sheets.

\begin{lem}\label{rewritinglemma}\cite[Lemma 5.4]{fujiie2018quantization}
Let $\hat x$ and $\check x$ be defined as above and in accordance with Definition \ref{def:sheets}. Let $x_0$ be any of the turning points $\alpha_l,\beta_l,\beta_r,\alpha_r$, and let $y$ be an amplitude base point.
Then
\begin{align*}
u^\pm(x;x_0,y)=-iu^\mp(\hat x;x_0,\hat y)=iu^\mp(\check x;x_0,\check y).
\end{align*}
\end{lem}

\subsection{Symmetry}\label{ss:symmetry}

For constants $c=c(\lambda)$ depending on the spectral parameter $\lambda$ we shall simply write $c(\mu)$ with the convention that $\mu$ is always defined via $\lambda=i\mu$. We then write $c(\bar\mu)$ to represent the value of $c$ at the reflection of $\lambda$ in the imaginary axis, i.e., at $i\bar\mu=-\bar\lambda$. We let
$$
c^\myconj(\mu)=\overline{c(\bar\mu)}.
$$
Similarly, for functions $f(x)=f(x;\lambda)$ we simply write $f(x;\mu)$, and let $f^\myconj$ denote the function
$$
f^\myconj(x;\mu)=\overline{f(\bar x;\bar\mu)}.
$$
For a WKB solution $u(x;x_0(\mu),y,\mu)$ depending also on phase base point $x_0(\mu)$ and amplitude base point $y$ independent of $\mu$, we thus have
$$
u^\myconj(x;x_0(\mu),y,\mu)=\overline{u(\bar x;x_0(\bar\mu),y,\bar\mu)}.
$$

Recall that we fixed a determination of $H(z(x))$ so that if $\mu\in\R$ then at $\tilde x=(\alpha_l+\beta_l)/2\in\R$ we have
$$
H(z(\tilde x))=\bigg(\frac{V(\tilde x)+\mu}{V(\tilde x)-\mu}\bigg)^{1/4}>0.
$$
It is straightforward to check that for $\mu\in\R$, this determination implies that
\begin{align}\label{eq:nearleftlobe}
H(z(x))&\in\R_+\quad\text{when }\alpha_l<x<\beta_l,\\
H(z(x))&\in e^{i\pi/4}\R_+\quad\text{when } x<\alpha_l,\quad \beta_l<x<\beta_r,\quad \alpha_r<x,
\label{eq:awayfromlobes}
\end{align}
while 
\begin{equation}\label{factori}
H(z(x))\in\begin{cases}\R_+& \text{when $V(\beta_l)=V(\beta_r)>0$}\\
i\R_+&\text{when $V(\beta_l)=-V(\beta_r)>0$}\end{cases}
\quad\text{and }\beta_r<x<\alpha_r.
\end{equation}
When $V(\beta_l)=V(\beta_r)>0$ this is in accordance with the fact that
\begin{equation}\label{HA1}
H(z(-x))=H(z(x))\quad \text{if $V(x) = V(-x)$ for $x\in\R$.}
\end{equation}
When $V(x) = V(-x)$ for $x\in\R$, we have
$$
H(z(-x))=\bigg(\frac{-(V(x)-\mu)}{-(V(x)+\mu)}\bigg)^{1/4}=c/H(x)
$$
for some constant $c$. Using the determination above we find that for $\mu\in\R$ and $x<\alpha_l$,
$$
e^{i\pi/4}\R_+\ni H(-x)=c/H(x) \in ce^{-i\pi/4}\R_+
$$
which implies that $c=i$. The same conclusion can also be drawn from the observation that if $\mu\in\R$ and $\alpha_l<x<\beta_l$ then 
$$
i\R_+\ni H(-x)=c/H(x) \in c\R_+
$$
by \eqref{factori}, again showing that $c=i$, that is,
\begin{equation}\label{HA2}
H(z(-x))=i/H(z(x))\quad \text{if $V(x) = -V(-x)$ for $x\in\R$.}
\end{equation}
These observations will be used to prove two symmetry properties: one with respect to reflection of the spectral parameter in the imaginary axis, and one with respect to parity.

\begin{prop}\label{proposition:sym1}
Let $\mu\in B_\ve(\mu_0)$ and let $x_0(\mu)\in\C$ be a solution to $V(x)^2-\mu^2=0$. Then $x_0(\mu)=\overline{x_0(\bar\mu)}$. Let $y$ be an amplitude base point independent of $\mu$. 
If $V$ is a single-lobe, or a double-lobe with $V(\beta_l)=V(\beta_r)$ then near the lobe(s) we have
\begin{equation}\label{eq:leftlobe}
(u^\pm)^\myconj(x;\alpha_\bullet(\mu),y,\mu)= u^\mp(x;\alpha_\bullet(\mu),\bar y,\mu).
\end{equation}
If $V$ is a double-lobe with $V(\beta_l)=-V(\beta_r)$ then \eqref{eq:leftlobe} holds near the left lobe while
\begin{equation}\label{eq:rightlobe}
(u^\pm)^\myconj(x;\alpha_\bullet(\mu),y,\mu)=- u^\mp(x;\alpha_\bullet(\mu),\bar y,\mu)
\end{equation}
near the right lobe. 
In the Stokes region to the left of $\alpha_l$ or to the right of $\alpha_r$,
$$
(u^\pm)^\myconj(x;\alpha_\bullet(\mu),y,\mu)=iu^\pm(x;\alpha_\bullet(\mu),\bar y,\mu).
$$
\end{prop}

\begin{proof}
Since $V$ is real-analytic we have $\overline{V(\bar x)}=V(x)$, which implies that
$V(\overline{\alpha_l(\mu)})-\bar\mu=0$.
Since $\alpha_l(\bar\mu)$ also satisfies this equation it follows that $\overline{\alpha_l(\mu)}=\alpha_l(\bar\mu)$, for $\alpha_l(\mu_0)\in\R$ and the turning points depend continuously on $\mu\in B_\ve(\mu_0)$. Hence $\alpha_l^\ast(\mu)=\alpha_l(\mu)$. The same arguments show that $x_0^\ast(\mu)=x_0(\mu)$ when $x_0$ is any of the other three turning points.

Next, if $V$ is a single-lobe and $x$ lies in the domain between the turning points, or if $V$ is a double-lobe and $x$ lies in either the domain between the left pair or in the domain between the right pair of turning points, then
$z(x,\mu)=i\int(V^2-\mu^2)^{1/2}dt$ with real integrand when $x,\mu\in\R$.
It is then easy to check that $\overline{z(\bar x,\bar \mu)}=-z(x,\mu)$. (In particular, when $x$ and $\mu$ are real, $z(x,\mu)$ is purely imaginary, as expected.) If $V$ is a single-lobe or a double-lobe with $V(\beta_l)=\pm V(\beta_r)$ then, using \eqref{eq:nearleftlobe} or \eqref{factori}, one checks that $\overline{H(z(\bar x,\bar\mu))}=H(z(x,\mu))$ near the left lobe and $\overline{H(z(\bar x,\bar\mu))}=\pm H(z(x,\mu))$ near the right lobe, which implies that 
$$
\overline{Q(z(\bar x,\bar\mu))}=c\cdot Q(z(x,\mu))\begin{pmatrix} 0 & 1\\ 1&0 \end{pmatrix},
$$
with $c=1$ near the left lobe, and $c=\pm 1$ near the right lobe, with sign determined according to $V(\beta_l)=\pm V(\beta_r)$.
Since $\overline{z'(\bar x,\bar\mu)}=-z'(x,\mu)$, inspection of the governing equations for the amplitude function $w^\pm(x,h;y,\mu)$ shows that
\begin{equation}\label{eq:symamplitude}
\overline{w^\pm(\bar x,h;y,\bar\mu)}=
w^\mp(x,h;\bar y,\mu)
\end{equation}
which gives \eqref{eq:leftlobe}--\eqref{eq:rightlobe}.

Finally, if $x$ lies in the domain left of $\alpha_l$ or right of $\alpha_r$ then
$z(x,\mu)=\int(\mu^2-V^2)^{1/2}dt$ with real integrand when $x,\mu\in\R$,
so $\overline{z(\bar x,\bar \mu)}=z(x,\mu)$. Using \eqref{eq:awayfromlobes} one checks that $\overline{H(z(\bar x,\bar\mu))}=-iH(z(x,\mu))$ and $\overline{Q(z(\bar x,\bar\mu))}=i Q(z(x,\mu))$.
Inspection of the governing equations for the amplitude function $w^\pm(x,h;y,\mu)$  shows that $\overline{w^\pm(\bar x,h;y,\bar\mu)}=
w^\pm(x,h;\bar y,\mu)$. This proves the last statement of the proposition and the proof is complete.
\end{proof}

\begin{prop}\label{proposition:sym5}
Let $\mu\in B_\ve(\mu_0)$ and let $x_0(\mu)\in\C$ be a solution to $V(x)^2-\mu^2=0$. 
If $V(x) = V(-x)$ for $x \in \R$, then
$$
u^\pm(-x;x_0(\mu),y,\mu)=-\begin{pmatrix*}[r] 0 & 1\\ 1&0 \end{pmatrix*}u^\mp(x;-x_0(\mu),- y,\mu).
$$
If $V(x) = -V(-x)$ for $x \in \R$, then
$$
u^\pm(-x;x_0(\mu),y,\mu)=\pm \begin{pmatrix*}[r] 0 & -1\\ 1&0\end{pmatrix*}u^\mp(x;-x_0(\mu),-y,\mu).
$$
\end{prop}

\begin{proof}
Since we are only concerned with symmetry with respect to $x\mapsto -x$ we will omit $\mu$ from the notation.
If $V(x) = V(-x)$ for $x\in\R$, a change of variables shows that $z(-x,x_0)=-z(x,-x_0)$.
Also $z'(x)=z'(-x)$ and $H(z(- x))=H(z(x))$ by \eqref{HA1}.
The governing equations for the amplitude function $w^\pm(x,h;y)$  imply that
\begin{equation*}
w^\pm(- x,h;y)=
w^\mp(x,h;- y).
\end{equation*}
Noting that $Q(z(- x))=Q(z(x))$ and
$$
-\begin{pmatrix} 0 & 1\\ 1&0 \end{pmatrix}\begin{pmatrix*}[r] 1 & 1\\ -1&1 \end{pmatrix*}Q(z(x))
=\begin{pmatrix*}[r] 1 & 1\\ -1&1 \end{pmatrix*}Q(z(x))\begin{pmatrix} 0 & 1\\ 1&0 \end{pmatrix}
$$
and that squaring the right-most matrix gives the identity, we obtain the first formula.

If $V(x)=-V(-x)$ for $x\in\R$ then $z$ satisfies the same relations as above while $H(z(- x))=i/H(z(x))$ by \eqref{HA2}.  
The governing equations for $w^\pm(x,h;y,\mu)$ now give
\begin{equation*}
w^\pm(- x,h;y)=\begin{pmatrix*}[r]1&0\\0&-1\end{pmatrix*}
w^\mp(x,h;- y),
\end{equation*}
while
$$
Q(z(-x))=\begin{pmatrix}-iH(z(x))&-iH(z(x))\\-1/H(z(x))&1/H(z(x))\end{pmatrix}.
$$
Since 
$$
\begin{pmatrix}0&1\\1&0\end{pmatrix}^{(1\pm 1)/2}\begin{pmatrix*}[r]1&0\\0&-1\end{pmatrix*}
=\pm \begin{pmatrix*}[r]0&-1\\1&0\end{pmatrix*}\begin{pmatrix}0&1\\1&0\end{pmatrix}^{(1\mp 1)/2}
$$
the second formula therefore follows by checking that
$$
\begin{pmatrix*}[r]1&1\\-1&1\end{pmatrix*}Q(z(-x))\begin{pmatrix*}[r]0&-1\\1&0\end{pmatrix*}
=\begin{pmatrix*}[r]0&-1\\1&0\end{pmatrix*}\begin{pmatrix*}[r]1&1\\-1&1\end{pmatrix*}Q(z(x))
$$
with $Q(z(-x))$ described above. This straightforward verification is left to the reader.
\end{proof}

\begin{prop}\label{proposition:sym2}
Let $\mu\in B_\ve(\mu_0)$.
Then $I^\myconj=I$, $I^\myconj_\bullet=I_\bullet$ and $J^\myconj=J$. If $V(x)=\pm V(-x)$  then $I_l=I_r$.
\end{prop}

\begin{proof}
We adapt the arguments in the proof of \cite[Lemma IV.2]{Hirota2017Real}.
Since 
$$
I_l(\bar\mu)=\int_{\alpha_l(\bar\mu)}^{\beta_l(\bar\mu)}(V(t)^2-\bar\mu^2)^{1/2}\de t
=\int_{\overline{\alpha_l(\mu)}}^{\overline{\beta_l(\mu)}}\overline{(V(\bar t)^2-\mu^2)^{1/2}}\de t
$$
by Proposition \ref{proposition:sym1}, a change of variables gives $I_l(\bar\mu)=\overline{I_l(\mu)}$, which proves that $I_l^\myconj=I_l$.
The same arguments show that $I^\myconj=I$, $I_r^\myconj=I_r$ and $J^\myconj=J$. If $V(x)=\pm V(-x)$ then $\alpha_l=-\alpha_r$ and $\beta_l=-\beta_r$, so the identity $I_l=I_r$ follows by a change of variables.
\end{proof}

We end this section with a result that will be used to determine the location of the reference points $\mu_k^\mathrm{sl}$ and $\mu_k^\mathrm{dl}$ mentioned in the introduction. In the statement, we let for brevity $I(\mu)$ denote either the action integral \eqref{actionintegralsimplelobe}, or one of the action integrals $I_l,I_r$ given by \eqref{actionintegraldoublelobe}. It will be convenient to allow an error term which can be made exponentially small for any fixed $h$.

\begin{lem}\label{lem:referencepoints}
Let $\mathcal I(\mu,h)$ satisfy $\mathcal I=\mathcal I^\ast$ and suppose that $\mathcal I(\mu,h)=I(\mu)+ha(\mu,h)+O(e^{-A/h})$ for any $A>0$, with $a,\partial a/\partial\mu=O(h)$ as $h\to0$, uniformly for $\mu\in B_\ve(\mu_0)$. Then there is an $h_0>0$ such that $\mu\mapsto\mathcal I(\mu,h)$ is injective in $B_\ve(\mu_0)$ for all $0<h\le h_0$. In particular, if $0<h\le h_0$ and $\mu_k(h)\in B_\ve(\mu_0)$ is a root of the equation $\mathcal I(\mu,h)=y_k(h)$ for some $y_k\in\R$, then $\mu_k\in\R$.
\end{lem}

\begin{proof}
Note that
$$
\partial I_l(\mu)/\partial\mu=-\mu\int_{\alpha_l(\mu)}^{\beta_l(\mu)}(V(t)^2-\mu^2)^{-1/2}\, dt
$$
since $\alpha_l$ and $\beta_l$ depend analytically on $\mu$ and are roots to $V(x)^2-\mu^2=0$. At $\mu_0\in\R$, this is a real integral with positive integrand. Hence, $I_l'(\mu_0)<0$, where prime denotes differentiation with respect to $\mu$, and we can ensure that $I_l'(\mu)\ne0$ for $\mu\in B_\ve(\mu_0)$ by choosing $\ve$ sufficiently small. The same arguments show that $I_r'(\mu), I'(\mu)\ne0$ for $\mu\in B_\ve(\mu_0)$.
Let $\mathcal J(\mu,h)= I(\mu)+ha(\mu,h)$, then $\mathcal J'(\mu,h)=I'(\mu)+O(h)$, where now $I'$ is any of the three derivatives just discussed, so $\mathcal J(\mu,h)$ is injective in $B_\ve(\mu_0)$ if $h$ is sufficiently small. The same must be true for $\mathcal I$, for if $\mathcal I(\mu_1)=\mathcal I(\mu_2)$ for some $\mu_1\ne\mu_2$, then $0=\mathcal J(\mu_1,h)-\mathcal J(\mu_2,h)+O(e^{-A/h})$. Letting $A\to\infty$ gives $\mathcal J(\mu_1,h)=\mathcal J(\mu_2,h)$, a contradiction.
By assumption we have $\mathcal I^\ast=\mathcal I$, so 
$$
\mathcal I(\bar \mu_k,h)=\overline{\mathcal I(\mu_k,h)}=\overline{y_k}=y_k=\mathcal I(\mu_k,h)
$$
since $y_k$ is real. Since $\mathcal I$ is injective, we conclude that $\mu_k=\bar \mu_k$.
\end{proof}

\section{Eigenvalues for a single-lobe potential}\label{section:single-lobe}

Here we suppose that $V$ is a single-lobe potential near $\mu_0$, and let $B_\ve(\mu_0)$ be a small neighborhood as described in connection with Definition \ref{def:single-lobe}. We will consider eigenvalues $\lambda=i\mu$ with $\mu\in B_\ve(\mu_0)$ with the purpose of deriving the quantization condition \eqref{intro:qcondsw} and proving Theorem \ref{pEV}. We ask the reader to recall the relevant Stokes geometry described in \S\ref{ss:stokesgeometry} and Figure \ref{fig:stokes_config_single}.

It is known that there are solutions $u_0=u_0(\mu)\in L^2(\R_+)$ and $v_0=v_0(\mu)\in L^2(\R_-)$ of \eqref{evp} (unique modulo constant factors) such that  $\lambda = i\mu$ is an eigenvalue of $P(h)$ if and only if $u_0=cv_0$ for some $c=c(\mu,h)$, thus
\begin{equation}\label{eq:wronskiancondition}
\mathcal{W}(u_0, v_0) = 0.
\end{equation}
As in the Schr{\"o}dinger case, this can be shown by following the program of Olver \cite{olver1997asymptotics} -- see \cite{hatzizisis2020semiclassical} for a detailed presentation in this direction.

\begin{rem}\label{rmk:symmetrichoice}
By modifying $u_0$ and $v_0$ if necessary we may assume that $u_0^\ast=iu_0$, $v^\ast_0=iv_0$ and $u_0=cv_0$ with $c=c^\ast$. Indeed, if $\mu\in B_\ve(\mu_0)$ then $\bar\mu\in B_\ve(\mu_0)$ since $\mu_0\in\R$, which implies that
$u_0^\ast(\mu)\in L^2(\R_+)$ and $v_0^\ast(\mu)\in L^2(\R_-)$ for $\mu\in B_\ve(\mu_0)$. Note that $u_0^\ast$ and $v_0^\ast$ also solve \eqref{evp}. Set $\widetilde u_0=\frac12(u_0-iu_0^\ast)\in L^2(\R_+)$ and $\widetilde v_0=\frac12(v_0/c^\ast-iv_0^\ast/c)\in L^2(\R_-)$. By uniqueness follows that $u_0^\ast$ is a multiple of $u_0$ and  $v_0^\ast$ is a multiple of $v_0$. (If it happens that $u_0^\ast = iu_0$ then $\widetilde u_0=u_0$ and $\widetilde v_0=v_0/c^\ast$. If it happens that $u_0^\ast = -iu_0$ we take $\widetilde u_0=u_0^\ast$ and $\widetilde v_0=v_0^\ast/c$ instead.) It is then easy to see that $\lambda=i\mu$ is an eigenvalue of $P(h)$ if and only if $\widetilde u_0= cc^\ast \widetilde v_0$. Since $\widetilde u_0^\ast=i\widetilde u_0$, $\widetilde v_0^\ast=i\widetilde v_0$ and $(cc^\ast)^\ast=cc^\ast$, the claim follows.
\end{rem}

To calculate the Wronskian \eqref{eq:wronskiancondition}, we shall follow \cite[\S2]{gerard1988precise} and first show that modulo an exponentially small error, $u_0$ and $v_0$ are each multiples of exact WKB solutions. 
Pick real numbers $x_l,\tilde x_l$ and $\tilde x_r,x_r$ such that $x_l<\tilde x_l < \re \alpha_l$ and $\re\alpha_r<\tilde x_r<x_r$, and pick $y$ in the upper half plane such that $\re \alpha_l<\re y<\re\alpha_r$, see Figure \ref{fig:stokesone}. These may be chosen independent of $\mu\in B_\ve(\mu_0)$ if $\ve$ is small enough. Define two pairs of linearly independent exact WKB 
solutions $u_{l},\widetilde u_l$ and $u_{r},\widetilde u_r$ by setting
\begin{align}\label{eq:leftGG}
u_l = u^{+}(x;\alpha_l,x_{l}), \quad u_{r} = u^{-}(x;\alpha_r,x_{r}),\\
\widetilde u_l = u^{-}(x;\alpha_l,\tilde x_{l}), \quad \widetilde u_{r} = u^{+}(x;\alpha_r,\tilde x_{r}),
\label{eq:rightGG}
\end{align}
with $u^\pm$ given by \eqref{exactWKB}. By (a slight modification of) the arguments of G\'erard and Grigis \cite[\S2.2]{gerard1988precise} we obtain the following representation formulas, where we use similar notation to make comparison easier.

\begin{figure}
	\centering
\includegraphics[scale=.85]{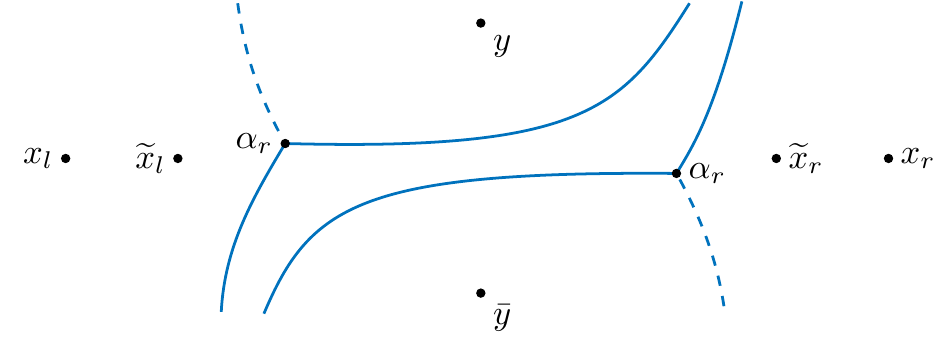}
\caption{\label{fig:stokesone} The location of amplitude base points relative the neighboring turning points for generic single-lobe potential $V$ and $\lambda=i\mu\in B_\ve(\lambda_0)$. Branch cuts are indicated by dashed lines.}
\end{figure}

\begin{lem}\label{lem:u0v0}
Let $u_\bullet,\widetilde u_\bullet$, $\bullet=l,r$, be given by \eqref{eq:leftGG}--\eqref{eq:rightGG}. Then
\begin{align}\label{v0}
v_0&=m_-(h)u_l+m_+(h)\widetilde u_l,\\
u_0&=l_-(h)u_r+l_+(h)\widetilde u_r,
\label{u0}
\end{align}
where $l_-(h)l_+(h)$ and $m_-(h)m_+(h)$ are of order $O(1)$ as $h\to0$, and $\lvert m_+(h)\rvert\le Ce^{z((x_l+\tilde x_l)/2)/h}$ and $\lvert l_+(h)\rvert\le Ce^{-z((x_r+\tilde x_r)/2)/h}$ are exponentially small as $h\to0$.
\end{lem}

Note that by Proposition \ref{proposition:sym1}, $u_r^\ast(x)=iu_r(x)$ when $x$ belongs to the Stokes domain containing $x_r$, and the same is true for $\widetilde u_r$. By Remark \ref{rmk:symmetrichoice} it is also true for $u_0$. In the domain containing $x_l$, the same relation holds for $u_l,\widetilde u_l,v_0$. A simple calculation then shows that
\begin{equation}\label{eq:lmsym}
l_\pm^\ast(\mu,h)=l_\pm(\mu,h),\quad  m_\pm^\ast(\mu,h)=m_\pm(\mu,h).
\end{equation}

Next, introduce two pairs $u^+_l,u^-_l$ and $u^+_r,u^-_r$ of intermediate exact WKB solutions given as
\begin{equation}\label{eq:upmbullet}
u_{\bullet}^{+}=u^{+}(x,h;\alpha_\bullet,y),\quad u_{\bullet}^{-}=u^{-}(x,h;\alpha_\bullet,\bar{y}),\quad \bullet=l,r.
\end{equation}
Note that
\begin{equation}\label{sol}
u_{l}^{\pm} = e^{\pm iI(\mu)/h} u_{r}^{\pm},
\end{equation}
where $I(\mu)$ is the action integral \eqref{actionintegralsimplelobe}.
Represent $v_0$ and $u_0$ as the linear combinations
\begin{align*}
	v_{0}=c_{11}u_{l}^{+} + c_{12}u_{l}^{-},\\
	u_{0}=c_{21}u_{r}^{+} + c_{22}u_{r}^{-},
\end{align*}
where the coefficients $c_{jk}$ depend on the parameters $\mu$ and $h$. For $x$ near the lobe we get
\begin{equation}\label{eq:calcsymleftlobe}
v_0^\ast(x)=c_{11}^\ast (u_{l}^{+})^\ast(x) + c_{12}^\ast (u_{l}^{-})^\ast(x)=
c_{11}^\ast u_{l}^{-}(x) + c_{12}^\ast u_{l}^{+}(x)
\end{equation}
by Proposition \ref{proposition:sym1}. Since $v_0^\ast=iv_0$ this means that $c_{12}=-ic_{11}^\ast$. Similarly, one checks that $c_{21}=-ic_{22}^\ast$. Hence,
\begin{align} \label{v0final}
	v_{0}&=c_lu_{l}^{+} -i c_l^\ast u_{l}^{-},\\
	u_{0}&=-ic_r^\ast u_{r}^{+} +c_ru_{r}^{-},
	\label{u0final}
\end{align}
for some symbols $c_l$ and $c_r$. 

\begin{lem}\label{lemma:clcr}
Let $\mu\in B_\ve(\mu_0)$. For any $A>0$ we may choose $x_r\gg\tilde x_r$ and $x_l\ll\tilde x_l$ so that the symbols $c_l$ and $c_r$ in \eqref{v0final}--\eqref{u0final} are given by 
\begin{align}\label{cl}
c_l&=m_-\tau_l,&&\tau_l=\tilde\tau_l-R_l,\\
c_r&=l_-\tau_r,&&\tau_r=\tilde\tau_r-R_r,
\label{cr}
\end{align}
where $\tau_\bullet,\tilde\tau_\bullet=1+O(h)$ and $R_\bullet=O(he^{-A/h})$ as $h\to0$, $\bullet=l,r$.
\end{lem}

\begin{proof}
Using \eqref{v0}--\eqref{u0} we see that
\begin{equation*}
c_l=m_-\frac{\mathcal W(u_l,u_l^-)}{\mathcal W(u_l^+,u_l^-)}+m_+\frac{\mathcal W(\widetilde u_l,u_l^-)}{\mathcal W(u_l^+,u_l^-)},\quad
c_r=l_-\frac{\mathcal W(u_r^+,u_r)}{\mathcal W(u_r^+,u_r^-)}+l_+\frac{\mathcal W(u_r^+,\widetilde u_r)}{\mathcal W(u_r^+,u_r^-)}.
\end{equation*}

For $\mathcal{W}(u_l, u_l^{-})$, $\mathcal{W}(u_l^{+}, u_l^{-})$, $\mathcal{W}(u_r^{+}, u_r)$ and $\mathcal{W}(u_r^{+}, u_r^-)$, we can directly apply the Wronskian formula \eqref{Wronskian1+}, and obtain
\begin{align*}
\mathcal{W}(u_l, u_l^{-}) & = 4iw_\mathrm{even}^+(\bar{y},h;x_l),  &  \mathcal{W}(u_r^{+}, u_r)&= 4iw_\mathrm{even}^+(x_r,h;y), \\
\mathcal{W}(u_l^{+}, u_l^{-}) & =  4iw_\mathrm{even}^+(\bar{y},h;y), &  \mathcal{W}(u_r^{+}, u_r^{-}) & =  4iw_\mathrm{even}^+(\bar{y},h;y).
\end{align*}
In particular, we can easily find curves such that each amplitude function appearing in these expressions has an asymptotic expansion described by Remark \ref{h-asymptotics}.
Indeed, this just requires being able to connect the relevant points (e.g., $x_l$ and $\bar y$ in $w^+_\mathrm{even}(\bar y,h;x_l)$) through curves along which $\re z(x)$ is increasing, which is clearly possible in view of the discussion connected to Figure \ref{fig:stokes_config_single} (see the figure for comparison). Hence, 
\begin{equation}\label{eq:tau}
\tilde\tau_l=\frac{\mathcal W(u_l,u_l^-)}{\mathcal W(u_l^+,u_l^-)}=\frac{w^+_\mathrm{even}(\bar y,h;x_l)}{w^+_\mathrm{even}(\bar y,h;y)},
\quad
\tilde\tau_r=\frac{\mathcal W(u_r^+,u_r)}{\mathcal W(u_r^+,u_r^-)}=\frac{w^+_\mathrm{even}(x_r,h;y)}{w^+_\mathrm{even}(\bar y,h;y)}
\end{equation}
have the stated asymptotic properties as $h\to0$ by Remark \ref{h-asymptotics}.

To compute $\mathcal W(\widetilde u_l,u_l^-)$ and  $\mathcal W(u_r^+,\widetilde u_r)$ we use the Wronskian formula \eqref{Wronskian2} for solutions of the same type instead of \eqref{Wronskian1+}. We then obtain
\begin{align*}
\mathcal W(\widetilde u_l,u_l^-)&=-4ie^{-2z(\tilde x_l;\alpha_l)/h}w^-_\mathrm{odd}(\tilde x_l,\bar y),\\
\mathcal W(u_r^+,\widetilde u_r)&= -4ie^{2z(\tilde x_r;\alpha_r)/h}w^+_\mathrm{odd}(\tilde x_r, y).
\end{align*}
We then obtain \eqref{cl}--\eqref{cr} by setting
\begin{equation*}
R_l=\frac{m_+}{m_-}e^{-2z(\tilde x_l;\alpha_l)/h}\frac{w^-_\mathrm{odd}(\tilde x_l,\bar y)}{w_\mathrm{even}^+(\bar y,y)},\quad R_r=\frac{l_+}{l_-}e^{2z(\tilde x_r;\alpha_r)/h}\frac{w^+_\mathrm{odd}(\tilde x_r, y)}{w_\mathrm{even}^+(\bar y,y)}.
\end{equation*}
Since $\re z(x)\to\pm\infty$ as $x\to\pm\infty$, the asymptotic properties of $R_\bullet$ follow from Lemma \ref{lem:u0v0} and Remark \ref{h-asymptotics}.
\end{proof}

These intermediate WKB solutions will allow us to prove the following quantization condition.

\begin{thm}\label{QCL_}
Suppose that $V$ is a single-lobe potential near $\mu_0$. Then, there exist positive constants $\ve$ and $h_{0}$, 	and a function $r(\mu, h)$ bounded on $B_\ve(\mu_{0}) \times (0, h_{0}]$  such that $\lambda = i\mu$, $\mu \in B_\ve(\mu_{0})$, is an eigenvalue of $P(h)$ for $h \in (0,h_{0}]$	if and only if 
\begin{equation}\label{eq:qcl}
I(\mu) = \bigg(k+\frac{1}{2}\bigg)\pi h + h^{2}r(\mu,h) 
\end{equation}
	holds for some integer $k$. Moreover, $r=r^\ast$.
\end{thm}

\begin{proof}
By \eqref{eq:wronskiancondition}, the quantization condition is $\mathcal{W}(v_{0}, u_{0}) = 0$. Using the representations \eqref{v0final}--\eqref{u0final} together with \eqref{sol} we see that $\mathcal{W}(v_{0}, u_{0}) = 0$ if and only if
\begin{align*}
0=(e^{iI(\mu)/h}c_{l}c_{r} +e^{-iI(\mu)/h}c_{l}^\ast c_{r}^\ast )\mathcal{W}(u_{l}^{+}, u_{l}^{-}).
\end{align*}
Since $u_{l}^+$ and $u_{l}^-$ are linearly independent we have $\mathcal{W}(u_{l}^{+}, u_{l}^{-}) \neq 0$, so the quantization condition is reduced to 
\begin{align*}
e^{2iI(\mu)/h}\bigg(\frac{c_lc_r}{c_l^\ast c_r^\ast}\bigg) = -1,
\end{align*}
that is, \eqref{eq:qcl} holds with
\begin{align} \label{remainder}
r(\mu,h) = -\frac{1}{2ih}\log\bigg(\frac{c_{l}c_{r}}{c_{l}^\ast c_{r}^\ast}\bigg) =-\frac{1}{2ih}\log\bigg(\frac{\tilde\tau_{l}\tilde\tau_{r}}{\tilde\tau_{l}^\ast\tilde\tau_{r}^\ast}\bigg)  +O(e^{-A/h})
\end{align}
where the second identity follows from 
an easy calculation using \eqref{eq:lmsym} and Lemma \ref{lemma:clcr}. By the same lemma we have $\tilde\tau_\bullet(\mu)=1+O(h)$ as $h\to0$, and this holds for all $\mu\in B_\ve(\mu_0)$. Since $\mu_0\in\R$ we have $\bar\mu\in B_\ve(\mu_0)$, so $\tilde\tau_\bullet^\ast(\mu)=\overline{\tilde\tau_\bullet(\bar\mu)}=1+O(h)$, and hence $r(\mu,h)=O(1)$ as $h\to0$.

To see that $r^\ast=r$, we use \eqref{remainder} and a logarithmic identity and get
$$
r(\mu,h) =\frac{1}{2ih}\log\bigg(\frac{c_{l} c_{r}}{c_{l}^\ast c_{r}^\ast}\bigg) ^{-1}= \frac{1}{2ih}\log\bigg(\frac{c_{l}^\ast c_{r}^\ast}{c_{l} c_{r}}\bigg) =r^\ast(\mu,h),
$$
which completes the proof.
\end{proof}

\begin{proof}[Proof of Theorem \ref{pEV}]
Let $r$ be given by Theorem \ref{QCL_}, and let $\tilde r$ be defined by the logarithm on the right of \eqref{remainder}, so that $r=\tilde r+O(e^{-A/h})$ for any $A>0$. Note that the amplitude functions $w^+_\mathrm{even}$ are so-called {\it analytic symbols} with respect to the spectral parameter $\lambda=i\mu$ and $h>0$. This means that $\partial w^+_\mathrm{even}(\mu)/\partial\mu=O(h)$ uniformly for $\mu\in B_\ve(\mu_0)$, see \cite{gerard1988precise} or \cite{sjostrand1982singularites}. Using the definition of $\tilde r$ together with \eqref{eq:tau} it is then easy to see that $h\partial \tilde r(\mu,h)/\partial \mu=O(h)$.

Let us define a function $\mathcal I(\mu,h)$ as
\begin{align*}
	\mathcal I(\mu,h) = I(\mu) - h^{2} r(\mu,h).
\end{align*}
Then $\mathcal I^\ast=\mathcal I$ so we may apply Lemma \ref{lem:referencepoints} (with $a$ in the lemma given by $a(\mu,h)=-h\tilde r(\mu,h)$) to conclude that if $h$ is sufficiently small then there is precisely one $\mu_k$ which solves $\mathcal I(\mu,h)=(k+\frac{1}{2})\pi h$. Moreover, $\mu_k\in\R$.
By Theorem \ref{QCL_} this means that eigenvalues $\lambda = i\mu$ of $P(h)$ are purely imaginary for $\mu$ near $\mu_0$. 
\end{proof}

\begin{rem}\label{rmk:EVnbh}
By Lemma \ref{lem:referencepoints} (with $a(\mu,h)\equiv0$) there is precisely one solution $\musw$ to $I(\mu)=(k+\frac{1}{2})\pi h$, and $\musw\in\R$. From the previous proof we then infer that $\lvert\mu_k-\musw\rvert=O(h^2)$ by the aid of Taylor's formula, where $\lambda_k=i\mu_k$ is the eigenvalue of $P(h)$ satisfying \eqref{eq:qcl}.
Moreover, similar arguments also show that
$$
\lvert \mu_j^\mathrm{sl}-\musw\rvert\ge \frac{1}{C}\lvert j-k\rvert\pi h,
$$
where $C$ is an upper bound of $\partial I(\mu)/\partial \mu$ for $\mu\in B_\ve(\mu_0)$. Hence,
if $\lambda_j=i\mu_j$ is an eigenvalue such that $\mu_j$ solves \eqref{eq:qcl} with $k$ replaced by $j\ne k$, then 
$$
\lvert \mu_j-\musw\rvert\ge\lvert\lvert\mu_j-\mu_j^\mathrm{sl}\rvert-\lvert \mu_j^\mathrm{sl}-\musw\rvert\rvert\ge O(h),
$$
showing that there is a unique eigenvalue $O(h^2)$-close to $\musw$.
\end{rem}

\begin{rmk}
As shown by Theorem \ref{pEV}, eigenvalues of $P(h)$ are purely imaginary for single-lobe potentials. In particular, the Stokes geometry depicted in the right panel of Figure \ref{fig:stokes_config_single} is never realized in the occasion of an eigenvalue. Heuristically this can be explained by the fact that there would otherwise be a curve transversal to the Stokes lines which connects the Stokes sector to the left of $\alpha_l$ with the sector to the right of $\alpha_r$. Hence, the exact WKB solution $u_l$ above, which can be written as $u_l(x)=e^{z(x)}\tilde u$ for some $\tilde u$, could be continued into this right sector along a curve where $\re z(x)$ is increasing. Letting $x\to\infty$ along $\R$ would yield a contradiction to the fact that $u_l$ is collinear (modulo an exponentially small error) with the function $u_0\in L^2(\R_+)$. 
\end{rmk}

\section{Eigenvalues for a double-lobe potential}\label{section:double-lobe}
In this section we suppose that $V$ is a double-lobe potential near $\mu_0$ and consider eigenvalues $\lambda=i\mu$ with $\mu\in B_\ve(\mu_0)$, where $B_\ve(\mu_0)$ is a small neighborhood as described in connection with Definition \ref{def:double-lobe}. The goal is to derive a quantization condition for such eigenvalues, which will then be used to prove the eigenvalue splitting occurring for symmetric potentials described in the introduction. For this reason, we will repeatedly include additional statements resulting from imposing the assumption that $V(x)=\pm V(-x)$ for $x\in\R$.

The Stokes geometry has been described in \S\ref{ss:stokesgeometry} and Figure \ref{fig:stokes_config_double}. As in Section \ref{section:single-lobe} we choose real numbers $x_l,\tilde x_l$ and $\tilde x_r,x_r$ such that $x_l<\tilde x_l<\re\alpha_l$ and $\re\alpha_r<\tilde x_r<x_r$. We also choose points $y_l$ and $y_r$ in the upper half-plane such that
$$
\re\alpha_l<\re y_l<\re\beta_l,\quad\re\beta_r<\re y_r<\re\alpha_r,
$$
see Figure \ref{figure4}. All points are chosen independent of $\mu\in B_\ve(\mu_0)$.
In the case when $V(x)=\pm V(-x)$ for $x\in\R$
we choose $y_\bullet$ and $x_\bullet,\tilde x_\bullet$ so that
$$
y_l=-\bar y_r,\quad\text{and}\quad x_l=-x_r,\quad \tilde x_l=-\tilde x_r.
$$

Let $u_0\in L^2(\R_+)$ and $v_0\in L^2(\R_-)$ be the functions described in Section \ref{section:single-lobe} such that $\lambda=i\mu$ is an eigenvalue of $P(h)$ if and only if $u_0=cv_0$ for some constant $c=c(\mu,h)$. We choose $u_0$ and $v_0$ in accordance with Remark \ref{rmk:symmetrichoice} so that $v_0^\ast=iv_0$, $u_0^\ast=iu_0$ and $c=c^\ast$. When $V$ is even we can choose $v_0$ as above and define 
\begin{equation}\label{eq:evendefu0}
u_0(x)=-\begin{pmatrix} 0&1\\1&0\end{pmatrix}v_0(-x).
\end{equation}
Then $u_0^\ast=iu_0$ and, using the fact that $v_0$ solves \eqref{evp}, it is easy to check that $u_0$ also solves \eqref{evp}. When $V$ is odd we instead define $u_0$ as
\begin{equation}\label{eq:odddefu0}
u_0(x)=-\begin{pmatrix*}[r]0&-1\\1&0\end{pmatrix*}v_0(-x),
\end{equation}
compare with Proposition \ref{proposition:sym5}.

\begin{figure}
\centering
\includegraphics[scale=.85]{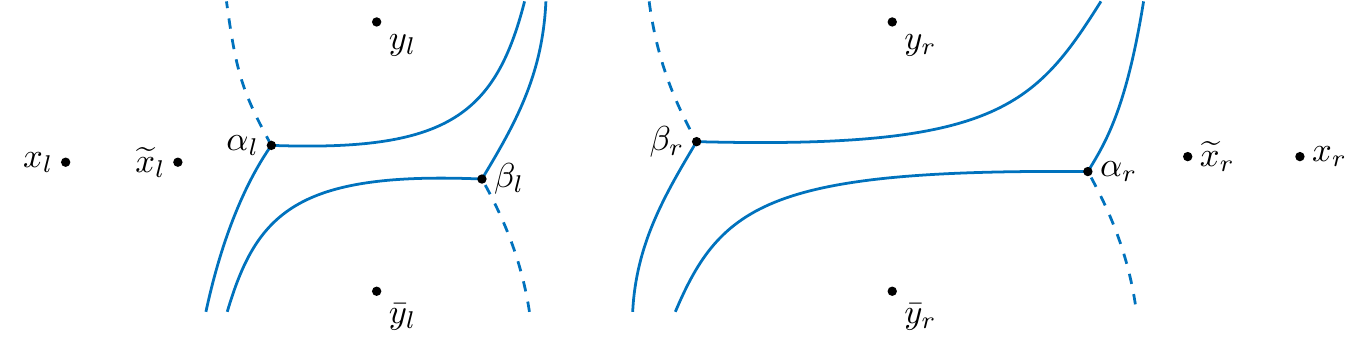}
\caption{\label{figure4} The location of amplitude base points relative the neighboring turning points for generic double-lobe potential $V$ and $\mu\in B_\ve(\mu_0)$. Branch cuts are indicated by dashed lines.}
\end{figure}

Introduce intermediate exact WKB solutions $u^+_\bullet=u^+(x;\alpha_\bullet,y_\bullet,\mu)$ and $u_\bullet^-=u^-(x;\alpha_\bullet,\bar y_\bullet,\mu)$, $\bullet=l,r$. Note that these are defined as in Section \ref{section:single-lobe} except that $y$ has now been replaced by $y_\bullet$, compare with \eqref{eq:upmbullet}. The reason for this is of course that we now have two lobes instead of one. Represent $v_0$ and $u_0$ as the linear combinations 
\begin{align*}
	v_{0}=c_{11}u_{l}^{+} + c_{12}u_{l}^{-},\\
	u_{0}=c_{21}u_{r}^{+} + c_{22}u_{r}^{-}.
\end{align*}
As for single-lobes one checks that $c_{12}=-ic_{11}^\ast$ using Proposition \ref{proposition:sym1} near the left lobe, see \eqref{eq:calcsymleftlobe}.
When $V(\beta_l)=\pm V(\beta_r)$ we use Proposition \ref{proposition:sym1} near the right lobe and obtain
$$
u_0^\ast(x)=c_{21}^\ast (u_{r}^{+})^\ast(x) + c_{22}^\ast (u_{r}^{-})^\ast(x)=
\pm c_{21}^\ast u_{r}^{-}(x) \pm c_{22}^\ast u_{r}^{+}(x).
$$
Since $u_0^\ast=iu_0$ this means that $c_{21}=\mp ic_{22}^\ast$. Hence,
\begin{align} \label{v0finaldouble}
	v_{0}&=c_lu_{l}^{+} -i c_l^\ast u_{l}^{-},\\
	u_{0}&=\mp ic_r^\ast u_{r}^{+} +c_ru_{r}^{-},
	\label{u0finaldouble}
\end{align}
for some symbols $c_l$ and $c_r$.

\begin{rem}\label{rem:clequalscr}
In the symmetric case when $V(-x)=\pm V(x)$ we have $c_l=c_r$. Indeed, then $\alpha_l=-\alpha_r$ and by definition we have $y_l=-\bar y_r$. If $V$ is even then
$$
u_l^\pm(-x)=-\begin{pmatrix}0&1\\1&0\end{pmatrix} u_r^\mp(x)
$$
by Proposition \ref{proposition:sym5}. In view of the definition \eqref{eq:evendefu0} of $u_0$ for even $V$ we get
$$
u_0(x)=-\begin{pmatrix}0&1\\1&0\end{pmatrix}(c_lu_l^+(-x)-ic_l^\ast u_l^-(-x))=c_l u_r^-(x)-ic_l^\ast u_r^+(x).
$$
Comparing the right-hand side with \eqref{u0finaldouble} we see that $c_l=c_r$ since $V(\beta_l)=V(\beta_r)$ when $V$ is even. If $V$ is odd then
$$
u_l^\pm(-x)=\pm \begin{pmatrix*}[r]0&-1\\1&0\end{pmatrix*} u_r^\mp(x)
$$
by Proposition \ref{proposition:sym5}. In view of the definition \eqref{eq:odddefu0} of $u_0$ for odd $V$ we get
$$
u_0(x)=-\begin{pmatrix*}[r]0&-1\\1&0\end{pmatrix*}(c_lu_l^+(-x)-ic_l^\ast u_l^-(-x))= c_l u_r^-(x)+ ic_l^\ast u_r^+(x).
$$
We again see that $c_l=c_r$ in view of \eqref{u0finaldouble} since $V(\beta_l)=-V(\beta_r)$ when $V$ is odd.
\end{rem}

Let $u_l,\widetilde u_l$ and $u_r,\widetilde u_r$ be given by \eqref{eq:leftGG}--\eqref{eq:rightGG}. Then Lemma \ref{lem:u0v0} holds also for double-lobe potentials, and by replacing $y$ with $y_\bullet$ in the proof of Lemma \ref{lemma:clcr} we find that the symbols $c_l,c_r$ in \eqref{v0finaldouble}--\eqref{u0finaldouble} can also be written as in \eqref{cl}--\eqref{cr}, that is, 
\begin{align}\label{cldouble}
c_l&=m_-\tau_l,&&\tau_l=\tilde\tau_l-R_l,\\
c_r&=l_-\tau_r,&&\tau_r=\tilde\tau_r-R_r,
\label{crdouble}
\end{align}
where $\tau_\bullet,\tilde\tau_\bullet=1+O(h)$ and $R_\bullet=O(he^{-A/h})$ for any $A>0$ as $h\to0$, $\bullet=l,r$.
When $V$ is symmetric we have $u_l(-x)=u_r(x)$ and $\widetilde u_l(-x)=\widetilde u_r(x)$, so using the representations \eqref{v0}--\eqref{u0} one checks that $m_-=l_-$ (while $m_+=\pm l_+$ when $V(x)=\pm V(-x)$) as in Remark \ref{rem:clequalscr}. Since $c_l=c_r$ it follows that $\tau_l=\tau_r$.

It will be convenient to introduce WKB solutions $u_1,u_2,u_3,u_4$
defined, when $V(\beta_l)=\pm V(\beta_r)$, by
\begin{align}
\label{u1u2}
u_1(x)&=i\tau_lu^+(x;\beta_l,y_l),
&& u_2(x)=-\tau_l^\myconj u^-(x;\beta_l,\bar y_l),\\
u_3(x)&=i\tau_r u^-(x;\beta_r,\bar y_r),
&& u_4(x)=\mp \tau_r^\myconj u^+(x;\beta_r, y_r).
\label{u3u4}
\end{align}
Note that $u_1^\ast=iu_2$ and $u_3^\ast =iu_4$ by Proposition \ref{proposition:sym1}.
A simple calculation shows that
\begin{align*}
u_l^+(x)&=e^{iI_l/h}u^+(x;\beta_l,y_l),&& u_l^-(x)=e^{-iI_l/h}u^-(x;\beta_l,\bar y_l),\\
u_r^-(x)&=e^{iI_r/h}u^-(x;\beta_r,\bar y_r),&& u_r^+(x)=e^{-iI_r/h}u^+(x;\beta_r, y_r),
\end{align*}
where $I_l$ and $I_r$ are the action integrals given by \eqref{actionintegraldoublelobe}.
In view of \eqref{v0finaldouble}--\eqref{u0finaldouble} we get
\begin{align}\label{u_l}
v_0/m_-&=-ie^{i I_l/h}u_1+ie^{-i I_l/h} u_2,\\
u_0/l_-&=-ie^{i I_r/h}u_3+ie^{-i I_r/h} u_4.
\label{u_r}
\end{align}

Inspired by the analysis in \cite{mecherout2016pt} we define the central solutions 
\begin{equation}\label{central}
v_l=e^{-J/h}\frac{1}{2i}(u_3-u_4),\quad v_r=e^{-J/h}\frac{1}{2i}(u_1-u_2).
\end{equation}
Since $J^\ast=J$ by Proposition \ref{proposition:sym2} we have $v_l^\ast=iv_l$ and $v_r^\ast=iv_r$.

\begin{lem}\label{lem:centrallinindep}
The central solutions $v_l$ and $v_r$ are linearly independent if $h$ is sufficiently small.
\end{lem}

\begin{proof}
We prove linear independence by showing that the Wronskian of $v_l$ and $v_r$ is nonzero for small $h$. By \eqref{central} we have
$$
\mathcal W(v_l,v_r)=-\frac{e^{-2J/h}}{4}\big(\mathcal W(u_3,u_1)-\mathcal W(u_4,u_1)-\mathcal W(u_3,u_2)+\mathcal W(u_4,u_2)\big),
$$
so an application of Corollary \ref{cor:Wronskians} gives
\begin{align*}
\mathcal W(v_l,v_r)&=-ie^{-J/h}\big(\tau_l\tau_rw_\mathrm{even}^+(\bar y_r,y_l)+\tau_l\tau_r^\myconj w_\mathrm{even}^+(\check y_l,y_r)\\&\quad\quad\quad\quad\quad\quad+\tau_l^\myconj \tau_rw_\mathrm{even}^+(\bar y_r,\hat{\bar{y}}_l)+\tau_l^\myconj \tau_r^\myconj w_\mathrm{even}^+(\bar y_l,y_r)\big),
\end{align*}
where the expression in parenthesis is $4+O(h)$ as $h\to0$.
\end{proof}

Now write 
\begin{align}\label{v_ld}
v_l&=d_{11}u_1+d_{12}u_2,\\
v_r&=d_{21}u_3+d_{22}u_4.
\label{v_rd}
\end{align}
By Lemma \ref{lem:centrallinindep}, the relations \eqref{v_ld}--\eqref{v_rd} constitute an invertible change of basis, and a straightforward calculation yields
$$
\begin{pmatrix} u_1\\ u_2\end{pmatrix}=
D_l\begin{pmatrix} v_l\\ v_r\end{pmatrix},\quad D_l=\frac{1}{d_{11}+d_{12}}\begin{pmatrix} 1& 2i e^{J/h} d_{12} \\ 1& -2i e^{J/h} d_{11} \end{pmatrix}
$$
and 
$$
\begin{pmatrix} u_3\\ u_4\end{pmatrix}=
D_r\begin{pmatrix} v_l\\ v_r\end{pmatrix},\quad D_r=\frac{1}{d_{21}+d_{22}}\begin{pmatrix}  2i e^{J/h} d_{22} &1\\ -2i e^{J/h} d_{21} &1\end{pmatrix}.
$$
Recall that $\lambda=i\mu$ is an eigenvalue precisely when $\det (v_0\ u_0)=0$. Since $v_l$ and $v_r$ are linearly independent by Lemma \ref{lem:centrallinindep}, a straightforward computation using \eqref{u_l}--\eqref{u_r} shows that $\det (v_0\ u_0)=0$ is equivalent to 
$$
\det\begin{pmatrix} \begin{pmatrix}-ie^{iI_l/h} & ie^{-iI_l/h}\end{pmatrix} D_l \\  \begin{pmatrix}-ie^{iI_r/h}& ie^{-iI_r/h}\end{pmatrix} D_r\end{pmatrix}=0,
$$
i.e.,
\begin{align*}
0&=\frac{1}{i}\Big(e^{iI_l/h}-e^{-iI_l/h}\Big)
\frac{1}{i}\Big(e^{iI_r/h}-e^{-iI_r/h}\Big)\\
&\quad-4e^{2J/h}\Big(e^{iI_l/h}d_{12}+e^{-iI_l/h}d_{11}\Big)
\Big( e^{iI_r/h}d_{22}+e^{-iI_r/h}d_{21}\Big).
\end{align*}
We rewrite this as 
\begin{align}\label{def:f}
0&=\Big(e^{iI_l/h}d_{12}+e^{-iI_l/h}d_{11}\Big)
\Big( e^{iI_r/h}d_{22}+e^{-iI_r/h}d_{21}\Big)\\&\quad
-e^{-2J/h}\sin(I_l/h)\sin(I_r/h).
\notag
\end{align}

\begin{lem}\label{lem:d_jk}
Let $\mu\in B_\ve(\mu_0)$. In the case when $V(\beta_l)=\pm V(\beta_r)$ we have
\begin{align*}
d_{12}&=-\gamma_l, &&
d_{11}= -\gamma_l^\myconj,\\
d_{22}&= \mp \gamma_r, &&
d_{21}= \mp\gamma_r^\myconj,
\end{align*}
where $\gamma_\bullet$ are $1+O(h)$ as $h\to0$ for $\bullet=l,r$.
In addition, if $V(x) = \pm V(-x)$ for $x\in\R$ then $\gamma_l=\gamma_r$.
\end{lem}

\begin{proof}
We first note that since $v_l^\ast=iv_l$, $u_1^\ast=iu_2$ we have $d_{11}^\ast=d_{12}$ in view of \eqref{v_ld}. Similarly, $d_{21}^\ast=d_{22}$ by \eqref{v_rd}. Using the arguments in Remark \ref{rem:clequalscr} it is easy to check that in the symmetric case $V(x)=\pm V(-x)$ we have $d_{12}=\pm d_{22}$.

We thus need to calculate $d_{12}$ and $d_{22}$, and begin with $d_{12}$. By \eqref{central} and \eqref{v_ld} we have
$$
d_{12}=\frac{\mathcal W(u_1,v_l)}{\mathcal W(u_1,u_2)}=\frac{e^{-J/h}}{2i}\frac{\mathcal W(u_1,u_3)-\mathcal W(u_1,u_4)}{\mathcal W(u_1,u_2)}.
$$ 
An application of Corollary \ref{cor:Wronskians} in the appendix therefore gives $d_{12}=-\gamma_l$ where
$$
\gamma_l=\frac{\tau_rw_\mathrm{even}^+(\bar y_r,y_l)+\tau_r^\myconj w_\mathrm{even}^+(\hat y_r,y_l)}{2\tau_l^\ast w_\mathrm{even}^+(\bar y_l,y_l)}
$$
so that $\gamma_l=1+O(h)$ as $h\to0$ by Remark \ref{h-asymptotics} and \eqref{cldouble}--\eqref{crdouble}.

Next, 
$$
d_{22}=\frac{\mathcal W(u_3,v_r)}{\mathcal W(u_3,u_4)}=\frac{e^{-J/h}}{2i}\frac{\mathcal W(u_3,u_1)-\mathcal W(u_3,u_2)}{\mathcal W(u_3,u_4)}.
$$ 
Using Corollary \ref{cor:Wronskians} we then get
$$
d_{22}=\mp \frac{\tau_lw_\mathrm{even}^+(\bar y_r,y_l)+ \tau_l^\myconj w_\mathrm{even}^+(\bar y_l,\hat{\bar{y}}_r)}{ 2\tau_r^\myconj w_\mathrm{even}^+(\bar y_r,y_r)}.
$$
We see that $d_{22}=\mp \gamma_r$ where $\gamma_r=1+O(h)$ as $h\to0$. 
\end{proof}

Combining \eqref{def:f} and Lemma \ref{lem:d_jk} we obtain the following Bohr-Sommerfeld quantization condition.

\begin{thm}\label{qc4}
Assume that $V$ is a double-lobe potential near $\mu_0$ and that $V(\beta_l)=\pm V(\beta_r)$. Then then there exist positive constants $\ve,h_{0}$ and functions $\gamma_\bullet(\mu,h)$, $\bullet=l,r$, defined on $B_\ve(\mu_0) \times (0,h_{0}]$ with $\gamma_\bullet=1+O(h)$ as $h\to0$, such that if $\mu \in B_\ve(\mu_0)$ then $\lambda = i\mu$ is an eigenvalue if and only if
\begin{align*}
\Big(e^{iI_l/h}\gamma_l+e^{-iI_l/h}\gamma_l^\myconj\Big)
\Big( e^{iI_r/h}\gamma_r+e^{-iI_r/h}\gamma_r^\myconj\Big)
\mp e^{-2J/h}\sin(I_l/h)\sin(I_r/h)=0.
\end{align*}
If $V(x)=\pm V(-x)$ for $x\in\R$ then $\gamma_l=\gamma_r$.
\end{thm}

\section{Symmetric potentials}\label{section:symmetric}

The purpose of this section is to prove Theorem \ref{splitting}. When doing so it will be convenient to use the following alternative form of Theorem \ref{qc4}.

\begin{prop}\label{prop:qc4}
Assume that $V$ is a double-lobe potential near $\mu_0$ and that 
$V(\beta_l)=\pm V(\beta_r)$. Then
there exist positive constants $\ve,h_{0}$ and functions $\rho_l,\rho_r$ defined on $B_\ve(\mu_0) \times (0,h_{0}]$ such that $\lambda = i\mu$, $\mu \in B_\ve(\mu_0)$, is an eigenvalue of $P(h)$ for $h \in (0, h_0]$ if and only if
\begin{align*}
4\rho_l\rho_r\cos(\widetilde I_l/h)\cos(\widetilde I_r/h)\mp e^{-2J/h}\sin(I_l/h)\sin(I_r/h)=0,
\end{align*}
where $\widetilde I_\bullet=I_\bullet+h\theta_\bullet=I_\bullet+O(h^2)$ and $\rho_\bullet=1+O(h)$ satisfy $\widetilde I_\bullet=\widetilde I_\bullet^\myconj$ and $\rho_\bullet=\rho_\bullet^\myconj$, $\bullet=l,r$. If $V(x)=\pm V(-x)$ then $I_l=I_r$ and $\widetilde I_l=\widetilde I_r$.
\end{prop}

\begin{proof}
We adapt the arguments in \cite[pp.~878--879]{mecherout2016pt} to our situation. By Theorem \ref{qc4}, $\lambda=i\mu$ is an eigenvalue if and only if
\begin{align*}
\Big(e^{iI_l/h}\gamma_l+e^{-iI_l/h}\gamma_l^\myconj\Big)
\Big( e^{iI_r/h}\gamma_r+e^{-iI_r/h}\gamma_r^\myconj\Big)
\mp e^{-2J/h}\sin(I_l/h)\sin(I_r/h)=0,
\end{align*}
where $\gamma_\bullet,\gamma_\bullet^\myconj$ are $1+O(h)$ as $h\to0$ for $\mu\in B_\ve(\mu_0)$. Write
\begin{align*}
\gamma_l&=(\gamma_l\gamma_l^\myconj)^{1/2}(\gamma_l/\gamma_l^\myconj)^{1/2}=\rho_l e^{i\theta_l},\\
\gamma_l^\myconj&=(\gamma_l\gamma_l^\myconj)^{1/2}(\gamma_l/\gamma_l^\myconj)^{-1/2}=\rho_l e^{-i\theta_l}
\end{align*}
where we choose branches of the square roots and the logarithm in such a way that $\rho_l=1+O(h)$, $\theta_l=O(h)$. Since 
$$
\rho_l^\myconj=\Big((\gamma_l\gamma_l^\myconj)^{1/2}\Big)^\myconj=\rho_l,
$$
it follows that 
$$
\rho_le^{-i\theta_l}=\gamma_l^\myconj=(\gamma_l)^\myconj=(\rho_le^{i\theta_l})^\myconj =\rho_l^\myconj e^{-i\theta_l^\myconj}=\rho_l e^{-i\theta_l^\myconj}
$$
so $\theta_l=\theta_l^\myconj$. Treating $\gamma_r,\gamma_r^\myconj$ the same way we find that $\lambda=i\mu$ is an eigenvalue if and only if
\begin{align*}
\rho_l\rho_r\Big(e^{i\widetilde I_l/h}+e^{-i\widetilde I_l/h}\Big)
\Big(e^{i\widetilde I_r/h}+e^{-i\widetilde I_r/h}\Big)\mp e^{-2J/h}\sin(I_l/h)\sin(I_r/h)=0,
\end{align*}
where $\widetilde I_\bullet=I_\bullet+h\theta_\bullet=I_\bullet+O(h^2)$ satisfies $\widetilde I_\bullet^\myconj=\widetilde I_\bullet$. Moreover, if $V(x)=\pm V(-x)$ then $\widetilde I_l=\widetilde I_r$ since $\gamma_l=\gamma_r$ then. The result now follows by an application of Euler's formula.
\end{proof}

\begin{rem}\label{rmk:derivative}
Note that $\widetilde I_\bullet$ is injective near $\mu_0$ if $h$ is sufficiently small. Indeed, since the amplitude functions $w^+_\mathrm{even}$ are analytic symbols with respect to the spectral parameter $\lambda=i\mu$ and $h>0$, inspecting their definition we see that $\gamma_\bullet=\tilde\gamma_\bullet+O(he^{-A/h})$ for any $A>0$, where 
$\partial\tilde \gamma_\bullet(\mu)/\partial\mu=O(h)$ uniformly for $\mu\in B_\ve(\mu_0)$, see the proof of Theorem \ref{pEV}. Define $\tilde\theta_\bullet$ via $e^{i\tilde\theta_\bullet}=(\tilde\gamma_\bullet/\tilde\gamma_\bullet^\myconj)^{1/2}$. Then $\theta_\bullet=\tilde\theta_\bullet+O(he^{-A/h})$, and differentiating the identity $e^{i\tilde\theta_\bullet}=(\tilde\gamma_\bullet/\tilde\gamma_\bullet^\myconj)^{1/2}$ gives 
$$
\partial\tilde\theta_\bullet/\partial\mu=O(h),\quad h\to0,\quad\bullet=l,r,
$$
since $e^{i\tilde\theta_\bullet}=1+O(h)$ by Taylor's formula. The claim now follows from Lemma \ref{lem:referencepoints}.
\end{rem}

\begin{proof}[Proof of Theorem \ref{splitting}]
If $V(x)$ is either an even or an odd function of $x\in\R$ then $\widetilde I_l=\widetilde I_r$ by Proposition \ref{prop:qc4}. Omitting the indices $l$ and $r$, the proposition then implies that
\begin{equation}\label{evenoddqc}
\cos(\tilde I(\mu)/h )  \pm R(\mu,h)=0,
\end{equation}
where
\begin{equation} \label{eq:R}
R(\mu,h)=\begin{cases} (2 \rho(\mu))^{-1}e^{-J(\mu)/h}\sin(I(\mu)/h) &\text{when }V(x)=V(-x),\\
( 2i \rho(\mu))^{-1}e^{-J(\mu)/h}\sin(I(\mu)/h)&\text{when }V(x)=-V(-x).
\end{cases}
\end{equation}
Since $J(\mu_0)>0$ and $\im I_\bullet(\mu_0)=0$ by definition we can choose $\ve>0$ small enough to ensure that, say,
\begin{equation}\label{assum:ve1}
\re J(\mu)\ge \tfrac{1}{2}J(\mu_0),\quad \lvert\im I_\bullet(\mu)\rvert\le \tfrac{1}{4} J(\mu_0),\quad \mu\in B_\ve(\mu_0).
\end{equation}
Using \eqref{assum:ve1} it is not difficult to check that $R(\mu,h)=O(e^{-\frac{1}{4}J(\mu_0)/h})$ for $\mu\in B_\ve(\mu_0)$.
By \eqref{evenoddqc} we must then have $\tilde I(\mu)/h=x_k(h)+y_k$ where $y_k=(k+\frac{1}{2})\pi$ for some integer $k$ and $x_k(h)=\tilde I(\mu)/h-y_k$ is close to zero. Hence,
$$
\cos(\tilde I(\mu)/h ) = \cos(x_k(h)+y_k)=(-1)^{k+1}\sin(\tilde I(\mu)/h-y_k),
$$
which together with \eqref{evenoddqc} implies that
\begin{equation}\label{mu+-}
\tilde I(\mu)=\Big(k+\frac{1}{2}\Big)\pi h\pm(-1)^{k}h\arcsin(R(\mu,h))
\end{equation}
when $\mu$ satisfies \eqref{evenoddqc}. By Remark \ref{rmk:derivative} and Lemma \ref{lem:referencepoints}, $\tilde I=I+h\theta$ is injective near $\mu_0$ for all sufficiently small $h$, so that the roots $\mudw$ to $\tilde I(\mu)=(k+\frac{1}{2})\pi h$ are unique and real. Moreover, for such $h$ there are precisely two solutions of \eqref{mu+-} which are denoted by $\mu_k^\pm$, and since $\tilde I^\ast=\tilde I$ while $R^\ast=\pm R$ these solutions must satisfy
\begin{equation}\label{symmetricpoints}
\overline{\mu^\pm_k}=\mu^\pm_k\quad\text{or}\quad\overline{\mu^\pm_k}=\mu^\mp_k.
\end{equation}
By Taylor expanding $\tilde I(\mu^\pm_k)$ near $\mudw$ and using $R(\mu^\pm_k,h)=O(e^{-\frac{1}{4}J(\mu_0)/h})$ we obtain the auxiliary estimate $\lvert\mu_k^\pm-\mudw\rvert=O(he^{-\frac{1}{4}J(\mu_0)/h})$.

We now improve this estimate by noting that $J(\mu^\pm_k)=J(\mudw)+O(he^{-\frac{1}{4}J(\mu_0)/h})$ by Taylor's formula, which implies that
$$
e^{-J(\mu^\pm_k)/h}=e^{-J(\mudw)/h}(1+O(h^\infty)).
$$
Since $\tilde I(\mu) = I(\mu) + h\theta=I(\mu)+O(h^2)$, we find by \eqref{mu+-} that
$$
\sin\big(I(\mu_k^\pm)/h\big)=\sin(\tilde I(\mu_k^\pm)/h-\theta(\mu_k^\pm))
=(-1)^k(1+O(h^2)).
$$
Hence, by \eqref{eq:R} we have
\begin{equation}\label{eq:R2}
R(\mu^\pm_k,h)=c(-1)^ke^{-J(\mudw)/h}(1+O(h)) 
\end{equation}
where $c=\frac{1}{2}$ when $V$ is even and $c=\frac{1}{2i}$ when $V$ is odd. By Taylor expanding $\tilde I(\mu^\pm_k)$ near $\mudw$ and using \eqref{mu+-} and \eqref{eq:R2} we obtain the improved estimate $\lvert\mu_k^\pm-\mudw\rvert=O(he^{-J(\mudw)/h})$. 
On the other hand, 
\begin{align*}
\tilde I(\mu_k^\pm) &=\tilde I(\mudw) +\tilde I^{\prime}(\mudw)(\mu_k^\pm-\mudw)+O(h^2e^{-2J(\mudw)/h}) \\
&=\tilde I(\mudw) + I^{\prime}(\mudw)(\mu_k^\pm-\mudw)+O(h^2e^{-J(\mudw)/h}) \end{align*}
by Remark \ref{rmk:derivative}, and hence a straightforward computation gives
$$
\cos(\tilde I(\mu^\pm_k)/h ) = (-1)^{k+1} I'(\mudw)\frac {\mu_k^\pm-\mudw}{h}+O(he^{-J(\mudw)/h}). 
$$
By combining this identity with \eqref{evenoddqc} and \eqref{eq:R2} we obtain the asymptotic formulas (1) and (2) of Theorem \ref{splitting}. The final statement of the theorem is then an immediate consequence of taking complex conjugates of these formulas and applying the symmetry relations \eqref{symmetricpoints}.
\end{proof}

\appendix

\section{}

\subsection{Connection formulas between the left and right lobe}

Here we compute Wronskians between WKB solutions defined near the left and right lobes. The proofs of these results could be obtained by inspecting the proofs in \cite[Sec.~III]{Hirota2017Real} and \cite[Sec.~5]{fujiie2018quantization}, and are included here for the benefit of the reader. To shorten notation we mostly omit $h$-dependence from the expressions below. We always assume that $V(\beta_l(\mu_0))>0$ and that $\mu$ belongs to some small neighborhood $B_\ve(\mu_0)$. 
We shall often refer to the case when $V(\beta_l(\mu))=V(\beta_r(\mu))$ as case $1^\circ$ and $V(\beta_l(\mu))=-V(\beta_r(\mu))$ as case $2^\circ$, which is in accordance with the terminology used by Fujii{\'e} and Wittsten \cite{fujiie2018quantization}.

\begin{lem}\label{appendix:Wronskians1}
Let $u^+(x;\beta_\bullet,y_\bullet)$ and $u^-(x;\beta_\bullet,\bar y_\bullet)$, $\bullet=l,r$, be given by \eqref{exactWKB}.
Then 
$$
\mathcal W(u^+(x;\beta_\bullet,y_\bullet),u^-(x;\beta_\bullet,\bar y_\bullet))=4iw_\mathrm{even}^+(\bar y_\bullet;y_\bullet),\quad \bullet=l,r,
$$
where the amplitude function appearing on the right is $1+O(h)$ as $h\to0$ for $\mu\in B_\ve(\mu_0)$ and $\bullet=l,r$.
\end{lem}

\begin{proof}
According to the behavior of $\re z(x)$, 
there is a curve from $y_\bullet$ to $\bar y_\bullet$ along which $\re z(x)$ is strictly increasing. 
By evaluating the Wronskian at $\bar y_\bullet$ (see \eqref{Wronskian1+}) 
we get 
$$
\mathcal W(u^+(x;\beta_\bullet,y_\bullet),u^-(x;\beta_\bullet,\bar y_\bullet))=4iw^+_\mathrm{even}(\bar y_\bullet;y_\bullet),\quad \bullet=l,r,
$$
where $w^+_\mathrm{even}(\bar y_\bullet;y_\bullet)=1+O(h)$ by Remark \ref{h-asymptotics}.
\end{proof}

\begin{lem}\label{appendix:Wronskians2}
Let $u^+(x;\beta_\bullet,y_\bullet)$ and $u^-(x;\beta_\bullet,\bar y_\bullet)$, $\bullet=l,r$, be given by \eqref{exactWKB}.
Then in the case when $V(\beta_l)=\pm V(\beta_r)$ we have
\begin{itemize}
\item[$(\mathrm{i})$] $\mathcal W(u^+(x;\beta_l,y_l),u^-(x;\beta_r,\bar y_r))=4ie^{J/h}w_\mathrm{even}^+(\bar y_r;y_l)$,
\item[$(\mathrm{ii})$] $\mathcal W(u^+(x;\beta_l,y_l),u^+(x;\beta_r,y_r))=\pm i\cdot 4ie^{J/h}w_\mathrm{even}^+(\hat y_r;y_l)$,
\end{itemize}
where the amplitude functions appearing on the right are $1+O(h)$ as $h\to0$ for $\mu\in B_\ve(\mu_0)$.
\end{lem}

\begin{proof}
We start with the proof of (i) and note that the phase base points of $u^+(x;\beta_l,y_l)$ and $u^-(x;\beta_r,\bar y_r)$ differ. We therefore rewrite $u^+(x;\beta_l,y_l)$ as 
\begin{equation}\label{rewriteu_1}
u^+(x;\beta_l,y_l)=e^{J/h}u^+(x;\beta_r,y_l),
\end{equation}
where we have used the fact that
\begin{equation*}
\exp\bigg(i\int_{\beta_l}^{\beta_r}\sqrt{V(t)^2-\mu^2}\de t/h\bigg)=e^{J/h}.
\end{equation*}
This identity is straightforward to check; in fact it can be established using the proof of \cite[Lemma 5.5]{fujiie2018quantization} with obvious modifications. 
Since we can find a curve from $y_l$ to $\bar y_r$ along which $\re z(x)$ is strictly increasing we can evaluate the Wronskian at $\bar y_r$ (see \eqref{Wronskian1+}) which gives (i), with $w^+_\mathrm{even}(\bar y_r;y_l)=1+O(h)$ by Remark \ref{h-asymptotics}.

We now prove (ii).
By Lemma \ref{rewritinglemma} we have $u^+(x;\beta_r,y_r)=iu^-(\check x;\beta_r,\check y_r)$ for $\check x$ near $\check y_r$ in case $1^\circ$ and $u^+(x;\beta_r,y_r)=-iu^-(\hat x;\beta_r,\hat y_r)$ for $\hat x$ near $\hat y_r$ in case $2^\circ$. 
In each case, take the function on the right and continue it through the branch cut starting at $\beta_r$ into the domain in the usual sheet containing $y_l$. Note that at $y_l$ these functions take the values $iu^-(y_l;\beta_r,\check y_r)$ and $-iu^-(y_l;\beta_r,\hat y_r)$, respectively. 
Using \eqref{rewriteu_1} and evaluating the Wronskian at $\hat y_r$ (see \eqref{Wronskian1+}) gives
$$
\mathcal W(u^+(x;\beta_l,y_l),u^+(x;\beta_r,y_r))= 4ie^{J/h}\begin{cases}i w_\mathrm{even}^+(\check y_r;y_l)&\text{in case $1^\circ$},\\
-iw_\mathrm{even}^+(\hat y_r;y_l)&\text{in case $2^\circ$}.\end{cases}
$$
Since $z(\hat x)=z(\check x)$ the two amplitude functions on the right coincide. Also, since $\re z(\check x)$ ($\re z(\hat x)$) is a strictly increasing function of $\im\check x$ ($\im \hat x$) near $\check x=\check y_r$ ($\hat x=\hat y_r$), we can find a curve from $y_l$ to $\check y_r$ ($\hat y_r$), passing through the branch cut at $\beta_r$, along which $t\mapsto \re z(t)$ is strictly increasing. Hence, $w_\mathrm{even}^+(\check y_r;y_l)=1+O(h)$ as by Remark \ref{h-asymptotics}, which proves (ii).
\end{proof}

\begin{cor}\label{cor:Wronskians}
Let $u_1,u_2,u_3,u_4$ be the WKB solutions defined in \eqref{u1u2}--\eqref{u3u4}. Then in the case when $V(\beta_l)=\pm V(\beta_r)$ we have
\begin{itemize}
\item[$(\mathrm{i})$] $\mathcal W(u_1,u_3)=- 4ie^{J/h}\tau_l\tau_rw_\mathrm{even}^+(\bar y_r;y_l)$,
\item[$(\mathrm{ii})$] $\mathcal W(u_1,u_4)=4ie^{J/h}\tau_l\tau_r^\myconj w_\mathrm{even}^+(\check y_l;y_r)$,
\item[$(\mathrm{iii})$] $\mathcal W(u_3,u_2)=-4ie^{J/h}\tau_l^\myconj\tau_rw_\mathrm{even}^+(\bar y_r;\hat{\bar{y}}_l)$,
\item[$(\mathrm{iv})$] $\mathcal W(u_4,u_2)=4ie^{J/h}\tau_l^\myconj\tau_r^\myconj w_\mathrm{even}^+(\bar y_l;y_r)$,
\item[$(\mathrm{v})$] $\mathcal W(u_1,u_2)=-i\cdot 4i\tau_l\tau_l^\myconj w_\mathrm{even}^+(\bar y_l;y_l)$,
\item[$(\mathrm{vi})$] $\mathcal W(u_3,u_4)=\pm i\cdot 4i\tau_r\tau_r^\myconj w_\mathrm{even}^+(\bar y_r;y_r)$,
\end{itemize}
where all amplitude functions appearing on the right are $1+O(h)$ as $h\to0$ for $\mu\in B_\ve(\mu_0)$.
\end{cor}

\begin{proof}
$(\mathrm{i})$-$(\mathrm{ii})$ and $(\mathrm{v})$-$(\mathrm{vi})$ are immediate consequences of Lemmas \ref{appendix:Wronskians1} and \ref{appendix:Wronskians2} together with the definitions \eqref{u1u2}--\eqref{u3u4} of $u_1,u_2,u_3,u_4$.

Now note that $\mathcal W(u_4,u_2)=\mathcal W(-iu_3^\ast,-iu_1^\ast)=(\mathcal W(u_1,u_3))^\ast$. Let $c=c(\mu)$ be the constant $c(\mu)=w_\mathrm{even}^+(\bar y_r;y_l,\mu)$. Since $y_\bullet$ is independent of $\mu$ by construction (see Section \ref{section:double-lobe}) we have
$$
c^\ast(\mu)=\overline{c(\bar\mu)}=\overline{w_\mathrm{even}^+(\bar y_r;y_l,\bar \mu)}.
$$
By \eqref{eq:symamplitude} we then get
$$
c^\ast(\mu)=w_\mathrm{even}^-( y_r;\bar y_l,\mu)=w_\mathrm{even}^+(\bar y_l; y_r,\mu),
$$
where the last identity follows by inspecting the Wronskian formulas \eqref{Wronskian1+}--\eqref{Wronskian1-}. In view of $(\mathrm{i})$ this proves $(\mathrm{iv})$ since $J^\ast=J$. Since $\mathcal W(u_3,u_2)=(\mathcal W(u_1,u_4))^\ast$ and $\bar{\check{y}}=\hat{\bar{y}}$, it is easy to check that $(\mathrm{iii})$ follows from $(\mathrm{ii})$ in a similar manner.
\end{proof}

\section*{Acknowledgements}
We wish to express our gratitude to Setsuro Fujii{\'e} for many helpful discussions and valuable suggestions. We also thank the referees for helpful comments.
Jens Wittsten was supported by the Swedish Research Council grants 2015-03780 and 2019-04878.

\bibliographystyle{amsplain}
\bibliography{references}

\end{document}